\documentclass[11pt]{article}

\usepackage[utf8]{inputenc} 

\usepackage[margin=3.1cm]{geometry} 
\usepackage{color,graphicx}
\usepackage{amsmath}
\usepackage{amssymb,mathrsfs}
\usepackage{amsthm}
\usepackage{chemarr,textcomp}
\usepackage{esint}
\usepackage{cite}
\usepackage{booktabs}
\usepackage{array} 
\usepackage{paralist}
\usepackage{verbatim}
\usepackage{subfig} 
\usepackage{dsfont}
\usepackage{enumitem}
\usepackage{placeins}

\usepackage[colorlinks,linkcolor={black},citecolor={black},urlcolor={black}]{hyperref}

\usepackage[yyyymmdd,hhmmss]{datetime}

\usepackage{fancyhdr}
\usepackage{sectsty}
\allsectionsfont{\sffamily\mdseries\upshape} 
\usepackage[nottoc,notlof,notlot]{tocbibind} 
\usepackage[titles,subfigure]{tocloft}

\theoremstyle{plain}
\newtheorem{theorem}{Theorem}[section]
\newtheorem{corollary}[theorem]{Corollary}
\newtheorem{lemma}[theorem]{Lemma}
\newtheorem{proposition}[theorem]{Proposition}

\newtheorem{definition}[theorem]{Definition}

\newtheorem*{definition*}{Definition}

\theoremstyle{remark}
\newtheorem{remark}[theorem]{Remark}
\newtheorem{example}[theorem]{Example}
\newtheorem*{remark*}{Remark}
\newtheorem*{example*}{Example}
\newtheorem*{notation*}{Notation}
\newcommand{\ddd}{\mathrm{d}}

\numberwithin{equation}{section}

\newcommand{\one}{{{\bf 1}}}
\newcommand{\ip}[1]{\langle {#1}\rangle}
\newcommand{\emm}{m}

\newcommand{\bAC}{\mathbb{AC}}

\newcommand{\cCE}{\mathcal{CE}}
\newcommand{\cAC}{\mathcal{AC}}
\newcommand{\bal}{\mbox{isotropy }}

\newcommand{\bOmega}{{\overline{\Omega}}}

\DeclareMathOperator{\diam}{diam}
\DeclareMathOperator{\Lip}{Lip}

\DeclareMathOperator{\dive}{div}
\DeclareMathOperator{\conv}{conv}
\DeclareMathOperator{\supp}{supp}

\def\BS{\boldsymbol} 

          \def\bftheta{{\BS\theta}}
        
\def\bflambda{{\BS\lambda}}

\newcommand{\sd}{\mathsf{d}}

\newcommand{\bA}{\mathbb{A}}
\newcommand{\bW}{\mathbb{W}_2}
\newcommand{\bbW}{\mathbb{W}}

\newcommand{\cC}{\mathcal{C}}
\newcommand{\cK}{\mathcal{K}}
\newcommand{\cX}{\mathcal{X}}
\newcommand{\cM}{\mathcal{M}}
\newcommand{\cA}{\mathcal{A}}

\newcommand{\cT}{\mathcal{T}}
\newcommand{\cP}{\mathcal{P}}
\newcommand{\cQ}{Q}
\newcommand{\sQ}{\mathscr{Q}}
\newcommand{\cD}{\mathcal{D}}
\def\cI{{\mathcal I}}
\newcommand{\TV}{\rm TV}
\newcommand{\dphi}{\widetilde\phi}

\newcommand{\cW}{\mathcal{W}}
\newcommand{\bfn}{n}

\newcommand{\cU}{\mathcal{U}}
\newcommand{\hrho}{\widehat{\rho}}

\newcommand{\bphi}{\overline{\phi}}

\def\N{{\mathbb N}}
\def\Z{{\mathbb Z}}
\def\R{{\mathbb R}}

\def\bbT{{\mathbb T}}

\def\P{{\cP}}

\newcommand{\dd}{\, \mathrm{d}}
\newcommand{\eps}{\varepsilon}

\newcommand{\weakly}{\rightharpoonup}

\def\Xint#1{\mathchoice
{\XXint\displaystyle\textstyle{#1}}%
{\XXint\textstyle\scriptstyle{#1}}%
{\XXint\scriptstyle\scriptscriptstyle{#1}}%
{\XXint\scriptscriptstyle\scriptscriptstyle{#1}}%
\!\int}
\def\XXint#1#2#3{{\setbox0=\hbox{$#1{#2#3}{\int}$ }
\vcenter{\hbox{$#2#3$ }}\kern-.6\wd0}}

\def\dashint{\Xint-}

\definecolor{jan}{rgb}{0.0,0.3,0.8}

\title{Scaling limits of discrete optimal transport}

\author{Peter Gladbach\thanks{Mathematisches Institut, Universit\"at Leipzig, Augustusplatz 10, 04109 Leipzig, Germany (\texttt{gladbach@math.uni-leipzig.de})} 
\and Eva Kopfer\thanks{Institut f\"ur Angewandte Mathematik, Universit\"at Bonn, Endenicher Allee 60, 53115 Bonn, Germany (\texttt{eva.kopfer@iam.uni-bonn.de})}
\and Jan Maas\thanks{Institute of Science and Technology Austria (IST Austria),
Am Campus 1, 3400 Klosterneuburg, Austria (\texttt{jan.maas@ist.ac.at})}}
\date{}

\begin{document}
\maketitle
\begin{abstract}
We consider dynamical transport metrics for probability measures on discretisations of a bounded convex domain in $\R^d$.
These metrics are natural discrete counterparts to the Kantorovich metric $\bW$, defined using a Benamou--Brenier type formula.
Under mild assumptions we prove an asymptotic upper bound for the discrete transport metric $\cW_\cT$ in terms of $\bW$, as the size of the mesh $\cT$ tends to $0$.
However, we show that the corresponding lower bound may fail in general, even on certain one-dimensional and symmetric two-dimensional meshes.
In addition, we show that the asymptotic lower bound holds under an isotropy assumption on the mesh, which turns out to be essentially necessary. 
This assumption is satisfied, e.g., for tilings by convex regular polygons, and it implies Gromov--Hausdorff convergence of the transport metric.
\end{abstract}

\tableofcontents

\section{Introduction}

Over the last decades, optimal transport has become a vibrant research area at the interface of analysis, probability, and geometry.
A central object is the $2$-Kantorovich distance $\bW$ (often called $2$-Wasserstein metric), which plays a major role in non-smooth geometry and analysis, and in the theory of dissipative PDE. We refer to the monographs \cite{AmGiSa05,Vill09,Sant15} for an overview of the theory and its applications.

More recently, discrete dynamical transport metrics have been introduced in the context of Markov chains \cite{Maas11}, reaction-diffusion systems \cite{Miel11} and discrete Fokker--Planck equations \cite{CHLZ11}.
These metrics are natural discrete counterparts to $\bW$ in several ways: they have been used to obtain a gradient flow formulation for discrete evolution equations \cite{ErMa14,MaMa16}, and to develop a discrete theory of Ricci curvature that leads to various functional inequalities for discrete systems \cite{ErMa12,Miel13,FaMa16}.
The geometry of geodesics for these metrics is currently actively studied, both from an analytic point of view \cite{GaLiMo17,ErMaWi18}, and through numerical methods \cite{ERSS17,SoRuGuBu16}; see also \cite{chow2019discrete,li2018computations} for further recent developments involving discrete optimal transport.

It is natural to ask whether the discrete transport metrics converge to $\bW$ under suitable assumptions.
The first result of this type has been obtained in \cite{GiMa13}.
The authors approximated the continuous torus $\bbT^d$ by the discrete torus $\bbT_N^d = (\Z/N\Z)^d$, and endowed the space of probability measures $\cP(\bbT_N^d)$ with the discrete transport metric $\cW_N$. The main result in \cite{GiMa13} asserts that, under a natural rescaling, the metric spaces $(\cP(\bbT_N^d),\cW_N)$ converge to the $L^2$-Kantorovich space $(\cP(\bbT^d), \bW)$ in the sense of Gromov--Hausdorff as $N\to\infty$.

A different convergence result was subsequently obtained by Garcia Trillos \cite{Tril17}. Given a set $\cX_N$ consisting of $N$ distinct points in $\bbT^d$, Garcia Trillos considers the graph obtained by connecting all pairs of points that lie at distance less than $\eps$, for a suitable $\eps$ depending on $N$. Under appropriate conditions on the uniformity of the point set, it is shown in \cite{Tril17} that the discrete transport metric converges to $\bW$, provided that $\eps = \eps(N)$ decays sufficiently slow.
While the result of \cite{Tril17} covers a wide range of settings, the latter assumption typically implies that the number of neighbours of a point in $\cX_N$ tends to $\infty$ as $N \to \infty$; in particular, the result of \cite{GiMa13} is not contained in \cite{Tril17}.

The aim of this paper is to investigate Gromov--Hausdorff convergence for transport metrics on general finite volume discretisations of a bounded convex domain $\Omega \subseteq \R^d$. While our setting is different from \cite{Tril17}, it corresponds in terms of scaling to the limiting regime in which the results of \cite{Tril17} fail to apply.

\subsection*{Setting of the paper}

We informally present the main results of this paper. For precise definitions we refer to Section \ref{sec:preliminaries}.

Let $\Omega \subseteq \R^d$ be a bounded convex open set. We endow  the set of Borel probability measures $\cP(\Omega)$ with the $2$-Kantorovich metric, which can be expressed in terms of the \emph{Benamou--Brenier formula}
\begin{align*}
	\bW(\mu_0,\mu_1) =
		\inf\left\{\sqrt{\int_0^1\bA^*(\mu_t,\dot\mu_t)\dd t}\right\} \ ,
\end{align*}
for $\mu_0, \mu_1 \in \cP(\bOmega)$.
Here, the infimum is taken among all absolutely continuous curves $(\mu_t)_{t\in[0,1]}$ in $\cP(\bOmega)$ connecting $\mu_0$ and $\mu_1,$ and
\begin{align}\label{eq:action-cont}
	\bA^*(\mu,w)=\sup_{\phi \in \cC^1(\bOmega)}
		\bigg\{ 2 \ip{\phi,w} - \bA(\mu, \phi) \bigg\} \ ,
\qquad
	\bA(\mu, \phi) = \int_\bOmega|\nabla\phi|^2\dd\mu \ .
\end{align}

We discretise the domain $\Omega$ using a finite volume discretisation, closely following  the setup from \cite{EyGaHe00}. An \emph{admissible mesh} consists of a partition $\cT$ of $\bOmega$ into sets $K$ with non-empty and convex interior, together with a family of distinct points $\{x_K\}_{K \in \cT}$ such that $x_K \in \overline K$ for all $K \in \cT$.
We write $(K|L) = \overline K\cap \overline L$ to denote the flat convex surface with $(d-1)$-dimensional Hausdorff measure $|(K|L)|$. We make the geometric assumption that the vector $x_K - x_L$ is orthogonal to $(K|L)$ if $K$ and $L$ are neighbouring cells and we write $d_{KL}:= |x_K-x_L|$. In addition, we impose some mild regularity conditions on the mesh; see Definition \ref{def:admissible-mesh} for the notion of $\zeta$-regularity that is imposed in the sequel.
We write $[\cT] := \max_{K \in \cT} \diam(K)$ to denote the mesh size of $\cT$.

The discrete transport metric on $\cP(\cT)$ is defined in terms of a discrete Be\-na\-mou--Brenier formula: for $\emm_0, \emm_1 \in \cP(\cT)$, we set
 \begin{align*}
 	\cW_\cT(\emm_0, \emm_1)
 = \inf\left\{ \sqrt{ \int_0^1 \cA^*_\cT(\emm_t,\dot\emm_t) \dd t }
 	 \right\} \ ,
 \end{align*}
where the action functionals are defined using natural discrete counterparts to \eqref{eq:action-cont}:
\begin{align*}
	\cA_\cT^*(\emm,\sigma) &  = \sup_{\psi : \cT \to \R}
		\bigg\{ 2 \ip{\psi,\sigma} - \cA_\cT(\emm, \psi) \bigg\} \ ,
\\
	\cA_\cT(\emm, \psi) & = \frac12 \sum_{K,L \in \cT}
 	S_{KL}
	 \theta_{KL} \bigg(\frac{\emm(K)}{|K|}, \frac{\emm(L)}{|L|}\bigg)  \big( \psi(K) - \psi(L) \big)^2 \ .
\end{align*}
Here, the transmission coefficients $S_{KL}$ are defined by
$
	S_{KL} = \frac{|(K|L)|}{|x_K - x_L|}
$.
This choice ensures the formal consistency of the discrete and the continuous definitions; cf. Remark \ref{rem:consistency} below for a verification at the level of the associated Dirichlet forms. We refer to \cite[Theorem 4.2]{EyGaHe00} for a convergence result for the discrete heat equation to the continuous heat equation.

The functions $\theta_{KL} : \R_+ \times \R_+ \to \R_+$ are \emph{admissible means}, i.e., $\theta_{KL}$ is a continuous function that is $\cC^\infty$ on $(0,\infty) \times (0,\infty)$, positively $1$-homogeneous, jointly concave, and normalised (i.e., $\theta(1,1) = 1$); see Definition \ref{def:mean-function} below for further details. Furthermore, we impose the symmetry condition $\theta_{KL}(s,t) = \theta_{LK}(t,s)$.
A common choice is the logarithmic mean $\theta_{KL}(s,t) := \int_0^1 s^{1-p} t^p \dd p$ for all $K \sim L$, which naturally arises in the gradient flow structure of the discrete heat equation \cite{Maas11,Miel11,CHLZ11}.
We write $\bftheta = (\theta_{KL})$ to denote the collection of mean functions in the definition of $\cW_\cT$, and suppress the dependence of $\cW_\cT$ on $\bftheta$ in the notation.
The freedom to choose these mean functions is due to the discreteness of the problem. We will see that a careful choice is crucial in the sequel.

\subsection*{Statement of the main results}

The goal of this paper is to analyse the limiting behaviour of the discrete transport metrics $\cW_\cT$ as $[\cT] \to 0$. To formulate the main results we introduce the canonical projection operator $P_\cT : \cP(\bOmega) \to \cP(\cT)$ given by
\begin{align}\label{eq:proj-embed}
	\big(P_\cT \mu \big)(K) & = \mu(K) \qquad \text{for } \mu \in \cP(\bOmega) \text{ and } K \in \cT \ .
\end{align}

Our first main result establishes a one-sided asymptotic estimate for the discrete transport metric in great generality.

\begin{theorem}[Asymptotic upper bound for $\cW_\cT$]\label{thm:upper-bound-intro}
Fix $\zeta \in (0,1]$, and let $\mu_0, \mu_1 \in \cP(\bOmega)$. 
For any family of $\zeta$-regular meshes $\{\cT\}$, and for any choice of mean functions $\bftheta_\cT$, we have
\begin{align*}
	\limsup_{[\cT] \to 0} \cW_\cT(P_\cT \mu_0 , P_\cT \mu_1) \leq \bW(\mu_0,\mu_1) \ .
\end{align*}

\end{theorem}

In view of the Gromov--Hausdorff convergence results from \cite{GiMa13} and \cite{Tril17}, one might expect that a corresponding asymptotic lower bound for $\cW_\cT$ in terms of $\bW$ holds as well.
However, convergence can fail, even in one dimension, as the following example shows.

\begin{example*}[A; a one-dimensional period mesh]
For $N \in \N$ and $r \in (0,\frac12)$, we consider a one-dimensional discretisation $\cT_{r,N}$ of the unit interval $[0,1]$, obtained by alternatingly concatenating intervals of length $\frac{r}{N}$ and $\frac{1-r}{N}$.
\begin{figure}[h]
    \begin{center}
        \includegraphics[scale=1.0]{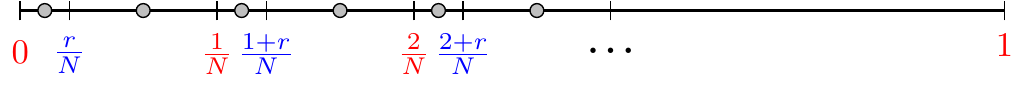}
    \end{center}
    \caption{The mesh $\cT_{r,N}$ on the interval $[0,1]$.}
\label{fig:1D}
\end{figure}
\end{example*}

The next result shows that $\cW_\cT$ can \emph{not} be bounded from below by $\bW$ as $[\cT] \to 0$.

\begin{proposition}[Counterexample to the lower bound for $\cW_\cT$]\label{prop:counterexamples}
Let $\Omega = (0,1)$ and let $\cT_{r,N}$ be as in Example (A)
above. Fix an admissible symmetric mean $\theta$, and consider the transport metric $\cW_\cT$ defined by setting $\theta_{KL} = \theta$ for all $K, L \in \cT$.
Then there exist probability measures $\mu_0, \mu_1 \in \cP(\bOmega)$ such that, for each fixed $r \in (0,\frac12)$,
\begin{align*}
    \limsup_{N \to \infty} \cW_{\cT_{r,N}}(P_{\cT_{r,N}} \mu_0, P_{\cT_{r,N}}\mu_1) < \bW(\mu_0,\mu_1) \ .
\end{align*}
\end{proposition}

We stress that the discrete heat flow converges to the continuous heat flow in the setting of this counterexample.

The idea behind the one-dimensional counterexample is that an ``unreasonably cheap'' discrete transport can be constructed by introducing microscopic oscillations in the discrete density in such a way that most of the mass is assigned to small cells. We refer to Section \ref{sec:counterexamples} for more details.

\medskip

In view of Proposition \ref{prop:counterexamples} it is natural to look for additional geometric assumptions on the mesh under which an asymptotic lower bound for $\cW_\cT$ in terms of $\bW$ can be obtained.
A \emph{weight function} on $\cT$ is a mapping $\bflambda : \cT \times \cT \to [0,1]$ satisfying $\lambda_{KL} + \lambda_{LK} = 1$ for all $K, L \in \cT$.
The following definition plays a central role in our investigations.

\begin{definition}[Asymptotic isotropy]
\label{def:asymp-intro}
A  family of admissible meshes $\{\cT\}$ is said to be \emph{asymptotically isotropic} with weight functions $\{\bflambda^\cT\}$ if, for all $K\in \cT$,
\begin{equation}\begin{aligned}\label{eq:asymptotic balance-intro}
 & \sum_{L} \lambda_{KL}^\cT 
 	\frac{| (K|L) |}{d_{KL}} (x_K - x_L) \otimes (x_K - x_L)
	 \leq |K| \big( 1  + \eta_{\cT}(K) \big)I_d \ ,
\end{aligned}\end{equation}
where $\sup_{K \in \cT} |\eta_{\cT}(K)| \to 0$ as $[\cT] \to 0$.
\end{definition}

The asymptotic isotropy condition puts a strong geometric constraint on a family of meshes, although it will be shown in Section \ref{sec:counterexamples} that isotropy always holds on average on a macroscopic scale.

\begin{remark*}[Centre-of-mass condition]
In examples it is often convenient to verify the asymptotic isotropy condition by checking the following stronger condition: we say that $\cT$ satifies the \emph{centre-of-mass condition} with weight function $\bflambda$ if 
\begin{align*}
	 \dashint_{(K|L)} x \dd S = \lambda_{LK} x_K + \lambda_{KL} x_L 
\end{align*}
for any pair of neighbouring cells $K, L \in \cT$.
As $\lambda_{KL} + \lambda_{LK} = 1$, this condition asserts that the centre of mass of the interface $(K|L)$ lies on the line segment connecting $x_K$ and $x_L$.
In the literature on finite volume methods this property is known as \emph{superadmissibility} of the mesh; see \cite[Lemma 2.1]{EyGaHe10}.

For all interior cells $K \in \cT$, we claim that the centre-of-mass condition yields the asymptotic isotropy condition \eqref{eq:asymptotic balance-intro} with equality and $\eta_{\cT}(K) = 0$. To see this, let $n = n(x)$ be the outward unit normal on $\partial K$, and note that
\begin{align}\label{eq:Gauss-identity}
    \int_{\partial K} (x - x_K) \otimes  n \dd S
    = |K| I_{d} \ ,
\end{align}
as can be shown by applying Gauss's Theorem to the vector fields $\Phi^{ij} : \R^d \to \R^d$ given by $\Phi^{ij}(x) := \ip{x - x_K, e_j} e_i$ for $1 \leq i, j \leq d$.
For all $L \sim K$, the centre-of-mass condition yields
\begin{align}
    \int_{(K|L)} (x - x_K) \otimes  n \dd S
    & = \int_{(K|L)} \lambda_{KL} (x_L - x_K) \otimes  n \dd S \nonumber
    \\& = \lambda_{KL} \frac{ | (K|L) |}{d_{KL}} (x_K - x_L) \otimes (x_K - x_L)\label{rhs com} \ .
\end{align}
The claim follows by summation over $L$.

Note that for boundary cells, the centre-of-mass condition does not imply asymptotic isotropy in general. If $\Omega$ is polygonal and $\cT$ can be extended to a global mesh satisfying the centre-of-mass condition, then our claim holds also for boundary cells by positive semi-definiteness of \eqref{rhs com}. This is the case in several examples below.
\end{remark*}

\begin{figure}[h]
	\begin{center}
        \includegraphics[scale=.195]{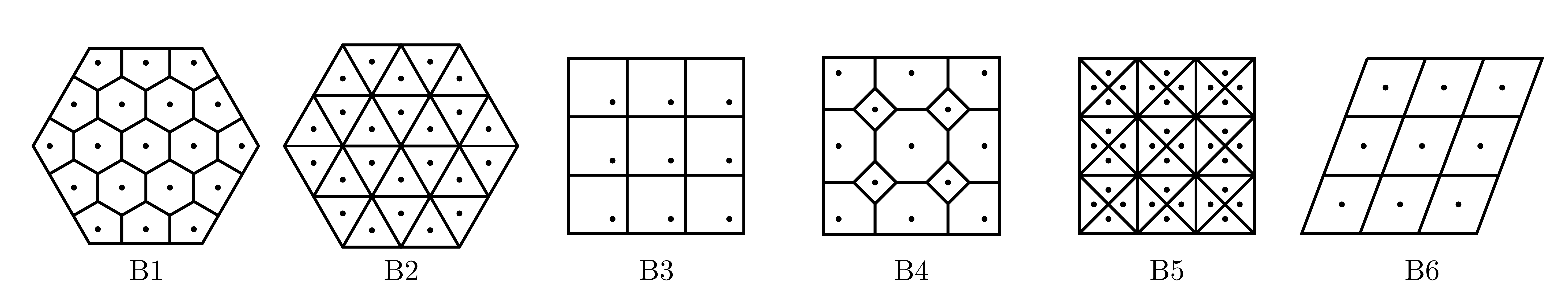}
    \end{center}
    \caption{
Five admissible meshes (B1)--(B5) and a non-admissible mesh (B6).
            }\label{fig:2D}
\end{figure}

\begin{example*}[A; revisited]
Clearly, the centre-of-mass condition holds in every one-di\-men\-sional mesh for an appropriate weight function $\bflambda$. For the one-dimensional periodic lattice $\cT_{r,N}$ in Example (A), it is immediately checked that $\lambda_{KL} = r$ if $K$ is small, and $\lambda_{KL} = 1-r$ if $K$ is large.
\end{example*}

\begin{example*}[B; Two dimensional meshes]\label{ex:com-domains}

The centre-of-mass condition holds for the regular hexagonal lattice (B1) and the  equilateral triangular lattice (B2) in dimension $2$, if the points $x_K$ are placed at the centre of mass of the cells. In these examples we have $\lambda_{KL} = \frac12$ for all $K$ and $L$. The hexagonal lattice (B1) is truncated in such a way that all interior interface have equal size. 
More generally, it is immediate to see that the centre-of-mass condition holds for any tiling of the plane by convex regular polygons; cf. \cite[Chapter 2]{GrSh87} for many examples.

The centre-of-mass condition is clearly satisfied for rectangular grids in any dimension, if the points $x_K$ are placed at the centre of the cells. The weights $\lambda_{KL}$ will depend on the size of the rectangles. It is possible to put the points $x_K$ at different positions, as is done in (B3). In that case, the centre-of-mass condition is violated, but the isotropy condition holds.

Another example for which the centre-of-mass condition holds is shown in (B4). The value of the weights $\lambda_{KL}$ is determined by the length ratio of the edges.

The centre-of-mass condition fails for the lattice in (B5). Indeed, to satisfy this condition, the points $x_K$ would have to be placed at the boundary of the cells, in a way that violates our assumption that the points $x_K$ are all distinct. This would lead to infinite transmission coefficients $S_{KL}$. The isotropy condition fails to hold as well, as will be discussed in Section 
\ref{sec:counterexamples}.

In each of the examples (B1)--(B4) it is readily checked that the isotropy condition also holds for boundary cells, by introducing suitable fake points outside of the domain.

The mesh in (B6) is not admissible, as the line segments connecting the points $x_K$ are not orthogonal to the cell interfaces. 
\end{example*}

Our next main result provides an asymptotic lower bound for $\cW_\cT$ in terms of $\bW$ under the assumption that the meshes $\{\cT\}$ satisfy the asymptotic isotropy condition, and the means $\theta_{KL}^\cT$ are carefully chosen to reflect this condition.

A mean function $(\theta_{KL})$ is said to be \emph{compatible} with a weight function $(\lambda_{KL})$ if, for any $K, L \in \cT$ and all $a , b \geq 0$, we have
\begin{align*}
	\theta_{KL}(a,b) \leq \lambda_{KL} a + \lambda_{LK} b \ , 
\end{align*}
or equivalently,  $\partial_1 \theta_{KL}(1,1) = \lambda_{KL}$ for any $K, L$; cf. Section \ref{sec:discrete-transport} for a more extensive discussion.

\begin{remark*}\label{rem:examples-means}
In the special case that $\lambda_{KL} = \frac12$, the compatibility condition holds for any admissible mean that is symmetric (i.e., $\theta_{KL}(s,t) = \theta_{KL}(t,s)$ for all $K, L$ and $s, t \geq 0$).
\end{remark*}

\begin{theorem}[Asymptotic lower bound for $\cW_\cT$]\label{thm:lower-bound-intro}
Fix $\zeta \in (0,1]$, and let $\mu_0, \mu_1 \in \cP(\bOmega)$. 
Let $\{\cT\}$ be a family of $\zeta$-regular meshes satisfying the asymptotic isotropy condition with weights $(\lambda_{KL}^\cT)_{K,L \in \cT}$, and let $(\theta_{KL}^\cT)_{K,L \in \cT}$ be admissible means that are compatible with $(\lambda_{KL}^\cT)_{K,L \in \cT}$. Then:
\begin{align*}
	\bW(\mu_0,\mu_1) \leq
		\liminf_{[\cT] \to 0} \cW_\cT(P_\cT \mu_0 , P_\cT \mu_1) \ .
\end{align*}
\end{theorem}

\begin{remark*}
As discussed above, the assumptions of the theorem are satisfied in the Examples (A) and (B1)--(B4).
\end{remark*}

In Section \ref{sec:G-H} we will prove slightly stronger versions of Theorems \ref{thm:upper-bound-intro} and \ref{thm:lower-bound-intro}, that provide uniform error bounds in terms of $\mu_0$ and $\mu_1$. As a consequence, we obtain the following result.

\begin{corollary}[Gromov--Hausdorff convergence of $\cW_\cT$]\label{cor:GH-conv}
Under the conditions of Theorem \ref{thm:lower-bound-intro}, we have convergence of metric space in the sense of Gromov--Hausdorff:
\begin{align*}
	(\cP(\cT), \cW_\cT) \to (\cP(\bOmega), \bW) \quad \text{ as } [\cT] \to 0 \ .
\end{align*}
\end{corollary}

Another consequence is the following result on the behaviour of $\cW_\cT$-geodesics.  
Let $\cQ_\cT : \cP(\cT) \to \cP(\bOmega)$ be the natural embedding defined in \eqref{eq:embedding} below.

\begin{corollary}[Convergence of geodesics]
\label{cor:min}
Under the conditions of Theorem \ref{thm:lower-bound-intro},
let $\mu_i \in \cP(\bOmega)$ and $\emm_i^\cT \in \cP(\cT)$ be such that $\cQ_\cT\emm_i^\cT \weakly \mu_i$ as $[\cT] \to 0$ for $i = 0, 1$. 
Then: 
\begin{align*}
	\lim_{[\cT] \to 0} \cW_\cT(\emm_0^\cT , \emm_1^\cT) 
		= \bW(\mu_0,\mu_1) \ .
\end{align*}
Moreover, if $(m_t^\cT)_{t \in [0,1]}$ is a constant speed geodesic in $(\cP(\cT), \cW_\cT)$ and $Q_\cT m_t^\cT \weakly \mu_t \in \cP(\bOmega)$ as $[\cT] \to 0$ for every $t\in [0,1]$, then $(\mu_t)_{t \in [0,1]}$ is a constant speed geodesic in $(\cP(\bOmega),\bW)$. 
\end{corollary}

Finally, we will show that the asymptotic isotropy condition is essentially necessary in Theorem \ref{thm:lower-bound-intro} and Corollary \ref{cor:GH-conv}.  More precisely, the following result shows that the asymptotic lower bound for $\cW_\cT$ fails to hold if the asymptotic isotropy condition is locally violated at all scales.
In this sense the asymptotic isotropy condition for $\{\cT\}$ is essentially equivalent to Gromov--Hausdorff convergence of $\cW_\cT$ to $\bW$.
The result relies crucially on the smoothness of the mean functions $\theta_{KL}$. In particular, the result does not apply to $\theta_{KL}(a,b) = \min\{a,b\}$.

\begin{theorem}[Necessity of asymptotic isotropy]\label{thm:necessity-intro}
Fix $\zeta\in(0,1]$, and let $\{\cT\}$ be a family of $\zeta$-regular meshes on $\Omega$.
For each $\cT$, let $\bftheta^\cT$ be a mean function on $\cT$ satisfying the regularity condition 
\begin{align*}
	\sup_{\cT} \sup_{K,L\in \cT}
	 \| D^2\theta_{KL}^\cT\|_{L^\infty(B((1,1),s))} < \infty
\end{align*}
for some $s > 0$.
Consider the weight functions $\bflambda^\cT$ defined by $\lambda_{KL} = \partial_1 \theta_{KL}^\cT(1,1)$, and assume that there exists a non-empty open subset $U \subseteq \Omega$, a unit vector $v\in S^{d-1}$ and $c > 0$, such that
\begin{align}
	\label{eq:anisotropy everywhere-intro}
\liminf_{[\cT]\to 0} \sum_{K\in \cT, K \subseteq V} 
	\bigg(\bigg(\sum_L \lambda_{KL} |n_{KL} \cdot v|^2 |(K|L)| d_{KL}\bigg) 
		- |K| \bigg)_+ \geq c|V| \ ,
\end{align}
for any non-empty open subset $V \subseteq U$.
Then there exist $\mu_0, \mu_1 \in \cP(\bOmega)$ such that
\begin{align*}
	\limsup_{[\cT]\to 0} \cW_\cT(P_\cT\mu_0,P_\cT\mu_1) < \bW(\mu_0,\mu_1) \ .
\end{align*}
\end{theorem}

The condition \eqref{eq:anisotropy everywhere-intro} can easily be verified in the setting of Examples (A) and (B5); cf. Section \ref{sec:counterexamples} for details.

\begin{remark*}
On a technical level, our method of proof offers the advantage that the maps for which Gromov--Hausdorff convergence is proved are the canonical projections $\cP_\cT$, rather than regularised versions of these maps, as in \cite{GiMa13}.
Another advantage is that we do not require any regularisation argument at the discrete level, as was done both in \cite{GiMa13} and \cite{Tril17}. All regularisation arguments are done at the continuous level.
In particular, we do not require any lower Ricci curvature bounds for the Markov chain at the discrete level, which would be quite restrictive.
 \end{remark*}

Let us finally discuss how the results in this paper relate to the discretisation of gradient flows. 
Indeed, it follows from results in \cite{Maas11,Miel11,CHLZ11} that the discrete heat equation on a mesh $\cT$ is the entropy gradient flow with respect to $\cW_\cT$ \emph{if and only if} the mean function $\theta_{KL}$ is the logarithmic mean for every pair of adjacent cells $K, L$. 
Moreover, for any vanishing sequence of regular meshes, solutions of the discrete heat equation converge to solutions of the continuous heat equation \cite{EyGaHe00}. 

By contrast, our main results imply that if $\theta_{KL}$ is the logarithmic mean, the associated transport metrics do \emph{not} converge to $\bW$, unless the meshes satisfy (rather restrictive) isotropy conditions.
Thus, in the non-isotropic setting, preservation of the gradient flow structure is incompatible with convergence of the associated transport distances.

For $1$-dimensional isotropic meshes, evolutionary $\Gamma$-convergence of the entropic gradient flow structure for the discrete heat flow with respect to the transport distance $\cW_\cT$ has been proved in \cite{DiLi15}. 
Loose speaking, this means that
\[
 \int_0^1 \cA_\cT^*(m_t, \dot m_t) + |\mathrm{grad}_{\cW_\cT}\mathrm{Ent}_{\cT}(m_t)|^2 \dd t\to \int_0^1 \bA^*(\mu_t, \dot \mu_t) + |\mathrm{grad}_{\bW}\mathrm{Ent}(\mu_t)|^2\dd t
\]
in the sense of $\Gamma$-convergence.
For gradient flow approximations to nonlinear parabolic problems, convergence results have been obtained as well; see, e.g., \cite{CaGu17}.

\subsection*{Structure of the paper}

In Section \ref{sec:preliminaries} we recall some known facts about the Kantorovich distance $\bW$ and the heat flow on convex bounded domains. Furthermore we introduce $\zeta$-regular meshes and the associated discrete transport distance $\cW_\cT$. 
Section \ref{sec:a-priori} contains several useful coarse \emph{a priori} bounds for discrete optimal transport, and in Section \ref{sec:finite-volume} we obtain finite volume estimates for the action functionals $\cA_\cT$ and $\cA_\cT^*$.
The failure of Gromov--Hausdorff convergence to $\bW$ (Proposition \ref{prop:counterexamples} and Theorem \ref{thm:necessity-intro}) is established in Section \ref{sec:counterexamples}.
Finally, Section \ref{sec:G-H} contains the proofs of the lower bound (Theorem \ref{thm:upper-bound-intro}), the upper bound (Theorem \ref{thm:lower-bound-intro}), and the convergence results (Corollaries \ref{cor:GH-conv} and \ref{cor:min}).

\small
\subsection*{Acknowledgements}

J.M. gratefully acknowledges support by the European Research Council (ERC) under the European Union's Horizon 2020 research and innovation programme (grant agreement No 716117), and by the Austrian Science Fund (FWF), Project SFB F65. E.K. gratefully acknowledges  support by the German Research Foundation through the Hausdorff Center for Mathematics and the Collaborative Research Center 1060.
She also thanks the Fields Institute for hospitality during the semester programme ``Geometric Analysis" in 2017. P.G. is funded by the Deutsche Forschungsgemeinschaft (DFG, German Research Foundation) -- 350398276.
We thank Arseniy Akopyan and Matthias Liero for useful comments on a draft of this paper.
We also thank the anonymous referees for careful reading and helpful comments.
\normalsize

\section{Preliminaries}
\label{sec:preliminaries}

\subsection{The Kantorovich metric}

Let $(\cX, \sd)$ be a Polish space, and let $\cP(\cX)$ be the set of Borel probability measures on $\cX$.
The class of Borel measures on $\cX$ is denoted by $\cM_+(\cX)$, and the class of signed Borel measures with mass $0$ by $\cM_0(\cX)$.

For $1 \leq p < \infty$, the \emph{$p$-Kantorovich metric} (often called $p$-Wasserstein metric) is defined by
\begin{align}\label{eq:Wp}
	\bbW_p(\mu_0, \mu_1)
	= \inf_{\gamma \in \Gamma(\mu_0, \mu_1)}
	 \bigg(\int_{\cX \times \cX} \sd(x,y)^p \dd \gamma(x,y) \bigg)^{1/p}
\end{align}
for $\mu_0, \mu_1 \in \cP(\cX)$.
Here, $\Gamma(\mu_0, \mu_1)$ denotes the set of all couplings (also called transport plans) between $\mu_0$ and $\mu_1$:
\begin{align*}
	\Gamma(\mu_0, \mu_1) = \{ \gamma \in \cP(\cX \times \cX) : \pi^0_\# \gamma = \mu_0 \text{ and } \pi^1_\# \gamma = \mu_1 \}  \ ,
\end{align*}
where $\pi^0(x,y) = x$, $\pi^1(x,y) = y$, and $\pi_\#^i \gamma$ denotes the push-forward of $\gamma$ under $\pi^i$. If $\cW(\mu_0, \mu_1) < \infty$, then the infimum in \eqref{eq:Wp} is attained; see, e.g., \cite{Vill09} for basic properties of $\bbW_p$.

\medskip

Let $\Omega \subseteq \mathbb R^d$ be a convex bounded open set.
Let $\cD=\cC_c^\infty(\mathbb R^d)$ be the space of test functions and let $\cD'$ be the space of distributions.
In this paper we will make use of the dynamical formulation of the metric $\bW$, which is given in terms of the action functional $\bA : \cM_+(\bOmega) \times \cC^1(\bOmega) \to \R$ and its Legendre dual $\bA^* : \cM_+(\bOmega) \times \cD' \to \R \cup \{ + \infty \}$, where
\begin{align}\label{eq:cont-action}
	\bA(\mu, \phi) := \int_{\bOmega} |\nabla \phi|^2 \dd \mu \ , \qquad
	 \bA^*(\mu, w)
		& = \sup_{\phi \in \cD} \bigg(2\ip{\phi, w} -  \bA(\mu, \phi) \bigg) \ .
\end{align}
The \emph{Benamou--Brenier formula} \cite{BeBr00} asserts that
\begin{equation}\begin{aligned}\label{eq:Benamou--Brenier}
	\bW(\mu_0, \mu_1)
	& = \inf \Bigg\{ \sqrt{\int_0^1 \bA^*(\mu_t, \dot \mu_t) \dd t} \ : \
				(\mu_t)_{t \in [0,1]} \in \bAC(\mu_0, \mu_1) \Bigg\} \ .
\end{aligned}\end{equation}
Here, $\bAC(\mu_0,\mu_1)$ denotes the class of all $\bW$-absolutely continuous curves $(\mu_t)_{t \in [0,1]}$ in $\cP(\bOmega)$ connecting $\mu_0$ and $\mu_1$.

For fixed $\delta > 0$, it will be useful to introduce the set $\cP_\delta(\bOmega)\subseteq \cP(\bOmega)$ consisting of all $\mu=u\dd x$ such that $u\colon \bOmega\to\mathbb R$ is Lipschitz continuous with $\Lip(u)\leq \frac1\delta$ and $\min_{x \in \bOmega} u(x) \geq \delta$.
For $\mu \in \cP_\delta(\bOmega)$ and $w\in L^2(\Omega)$  with $\int_\Omega w(x) \dd x = 0$, the maximiser in the definition of $\bA^*(\mu,w)$ is given by the unique distributional solution $\phi\in H^1(\Omega)$ to the elliptic problem
\begin{equation*}
		\begin{cases}
		- \dive (u \nabla \phi)  = w &\text{ in }\Omega\\
		\partial_\bfn \phi  = 0	&\text{ on }\partial\Omega
		\end{cases}
	\end{equation*}
satisfying $\int_\Omega\phi \dd x=0$.
 Moreover, $\bA(\mu,\phi) = \bA^*(\mu,w)$.

\medskip

Let $(H_a)_{a \geq 0}$ be the heat semigroup associated to the Neumann Laplacian $\Delta$ on $\Omega$.
Since $\Omega$ is convex, $\bOmega$ is a CD$(0,d)$ space in the sense of Bakry--\'Emery and Lott--Villani--Sturm. In particular, the heat semigroup satisfies the Bakry--\'Emery gradient estimate
\begin{align}\label{eq:Bakry-Emery}
	|\nabla H_a \phi|^2\leq H_a|\nabla \phi|^2
\end{align}
for all sufficiently smooth functions $\phi : \bOmega \to \R$,
and the local logarithmic Sobolev inequality
\begin{align}\label{eq:LSI-local}
  |\nabla \log H_a \phi|^2
 	\leq \frac{1}{a}\frac{H_a(\phi \log \phi) - H_a \phi \log H_a \phi}{H_a \phi} \ ,
\end{align}
for all smooth and positive functions $\phi : \bOmega \to \R$; cf. \cite[Theorem 5.5.2]{BaGeLe14}.
Moreover, it is well known that the Neumann heat kernel $h_a$ satisfies the Gaussian bounds
\begin{align}\label{eq: heat kernel estimates}
(4\pi a)^{-d/2}e^{-\frac{|x-y|^2}{4a}}
	\leq h_a(x,y)
	\leq C \big(a^{-d/2} \vee 1 \big) e^{-c\frac{|x-y|^2}{a}}
\end{align}
for all $a > 0$ and $x,y \in \bOmega$, with constants $c, C >0$ depending on $\Omega$; see \cite[Theorems 3.2.9 and 5.6.1]{Davi89}. The following result asserts that the heat semigroups maps $\cP(\bOmega)$ into $\cP_\delta(\bOmega)$.

\begin{lemma}\label{lem:P-delta}
For all $a > 0$ there exists a constant $\delta > 0$, depending on $\Omega$ and $a$, such that for any $\mu \in \cP(\bOmega)$ the density $u_a$ of $H_a \mu$ satisfies
\begin{align*}
	u_a(x) \geq \delta \text{ for all } x \in \bOmega
	\quad \text{ and } \quad
	\Lip(u_a)\leq \frac{1}{\delta} \ .
\end{align*}
\end{lemma}

\begin{proof}
The lower bound follows immediately from the Gaussian lower bound in \eqref{eq: heat kernel estimates}.

To prove the Lipschitz bound, we use the pointwise gradient inequality
\begin{align*}
	2a |\nabla H_a \phi |^2 \leq H_a (\phi^2) \ ,
\end{align*}
which follows from the CD$(0,d)$ property of $\Omega$; cf. \cite[Theorem 4.7.2]{BaGeLe14}. Together with \eqref{eq: heat kernel estimates} we obtain
\begin{align*}
	|\nabla u_a(x)|^2
		=    | \nabla H_{a/2}  u_{a/2} (x)|^2
		\leq a^{-1}  \| H_{a/2} (u_{a/2}^2)\|_\infty
		\leq a^{-1} \| u_{a/2} \|_{\infty}^2
		\leq C \ ,
\end{align*}
for some $C < \infty$ depending on $a$. This implies the result.
\end{proof}

\medskip

We collect some known properties of the Kantorovich metric that will be useful in the sequel.
Let us remark that the convexity of the domain $\Omega$ is crucial for Part \ref{item:heat-contraction} in the following lemma, as its proof relies on gradient estimates in the sense of Bakry and \'Emery (see, e.g., \cite[Theorem 3.17]{AmGiSa15} and \cite{BaGeLe14}).

\begin{lemma}[Bounds on the Kantorovich metric] \label{lem:Wass-facts}
The following assertions hold:
\begin{enumerate}[label=(\roman*)]\setlength\itemsep{0ex}
\item \label{item:heat-monotonicity}
(Monotonicity)
For $i = 0, 1$, let $\mu_i \in \cM_+(\bOmega)$ with $\mu_0 \geq \mu_1$, and let $w \in \cD'$. Then:
\begin{align}\label{eq:action-monotone}
	\bA^*(\mu_0, w) \leq \bA^*(\mu_1, w) \ .
\end{align}
\item \label{item:heat-contraction}
(Contraction bounds)
For any $\mu \in \cP(\bOmega)$, $\phi \in \cD$, $w \in \cD'$, and $a \geq 0$, we have
 \begin{align}\label{eq:action-heat}
	\bA(\mu, H_a \phi) \leq \bA(H_a \mu, \phi) \quad \text{and} \quad
	\bA^*(H_a \mu, H_a w) \leq \bA^*(\mu, w) \ .
\end{align}
Consequently, the following contraction property holds for any $\mu_0, \mu_1 \in \cP(\bOmega)$:
\begin{align}\label{eq:contraction-heat}
\bW(H_a \mu_0, H_a \mu_1) \leq \bW(\mu_0,\mu_1) \ .
\end{align}
\item \label{item:heat-holder}
(H\"older continuity of the heat flow)
There exists a constant $C < \infty$ such that
\begin{align*}
	\bW(\mu, H_a\mu) \leq C \sqrt{a}
\end{align*}
for any $\mu \in \cP(\bOmega)$ and $a \geq 0$.
\end{enumerate}
\end{lemma}

\begin{proof}
Assertion \ref{item:heat-monotonicity} follows immediately from the definitions.

The first part of Assertion \ref{item:heat-contraction} is a consequence of the Bakry--\'Emery gradient estimate \eqref{eq:Bakry-Emery} and the selfadjointness of $H_a$. The second part follows immediately, and the inequality \eqref{eq:contraction-heat} holds by the Benamou--Brenier formula \eqref{eq:Benamou--Brenier}.

To show Assertion \ref{item:heat-holder}, note first that $d\gamma(x,y) = h_a(x,y) \dd \mu(x) \dd y$ is a coupling between $\mu$ and $H_a \mu$.
Thus using \eqref{eq: heat kernel estimates} we obtain
\begin{equation}
\begin{aligned}\label{delta-heat}
	\bW(\mu, H_a\mu)^2
	& \leq \int_\bOmega \int_\bOmega |x-y|^2 h_a(x,y)\dd \mu(x) \dd y
	\\& \leq C \big(a^{-d/2} \vee 1 \big)
	\int_{\bOmega} \int_\bOmega |x-y|^2 e^{-c\frac{|x-y|^2}{a}} \dd\mu(x) \dd y
	\\& \leq C \big(a^{-d/2} \vee 1 \big) \int_{\bOmega} \int_{\R^d}
		|x-y|^2  e^{-c\frac{|x-y|^2}{a}}
			\dd y \dd\mu(x)
	\\& = C \big(1 \vee a^{d/2} \big) a \ .
\end{aligned}
\end{equation}
If $a \leq \diam(\Omega)^2$, we have $1 \vee a^{d/2} \leq C$ for some constant $C$ depending on $\Omega$, so we may absorb the factor $1 \vee a^{d/2}$ into the constant. If $a \geq \diam(\Omega)^2$, we trivially have $\bW(\mu, H_a \mu) \leq \diam(\Omega) \leq \sqrt{a}$.
\end{proof}

\subsection{Discrete transport metrics} \label{sec:discrete-transport}

We briefly recall the definition of the discrete transport metrics introduced in \cite{CHLZ11,Maas11,Miel11}.

Let $\cX$ be a finite set, and let $R : \cX \times \cX \to \R_+$ be a non-negative function satisfying $R(x,x) = 0$ for all $x \in \cX$. 
We interpret $R(x,y)$ as the transition rate from $x$ to $y$ for a continuous-time Markov chain on $\cX$, which we assume to be irreducible.
Under this assumption, there exists a unique invariant probability measure $\pi \in \cP(\cX)$. We assume that $\pi$ satisfies the \emph{detailed balance condition}:
\begin{align*}
	\pi(x) R(x,y) = \pi(y) R(y,x) \quad \text{ for all } x, y \in \cX \ .
\end{align*}

To define the discrete transport metric we need to choose a family of admissible means in the sense of the following definition. Note that these assumptions slightly differ from \cite[Assumption 2.1]{ErMa12}.

\begin{definition}[Admissible mean]
\label{def:mean-properties}
An \emph{admissible mean} is a continuous function $\theta :  \R_+ \times \R_+ \to \R_+$ that is $\cC^\infty$ on $(0,\infty) \times (0,\infty)$, positively $1$-homogeneous, jointly concave, and normalised (i.e., $\theta(1,1) = 1$).
\end{definition}

We collect some known properties of admissible means in the following result.

\begin{lemma}\label{lem:}
Let $\theta$ be an admissible mean.
\begin{enumerate}[label=(\roman*)]\setlength\itemsep{0ex}
\item \label{it:pos-mena}
	For $a, b \geq 0$ we have $\theta(a,b) \geq \min\{a, b\}$.
\item \label{it:hom-id} For any $a,b \geq 0$ and $c, d > 0$ we have
\begin{align}\label{eq:4-pt}
	\theta (a,b) \leq a \partial_1 \theta(c,d)  + b \partial_2 \theta(c,d) \ ,
\end{align}
with equality if $a = c$ and $b = d$.
\end{enumerate}
\end{lemma}

\begin{proof}
To show \ref{it:pos-mena}, assume first that $a > b$, and write
$(a,b) = \frac{b}{b + \eps} (b + \eps, b + \eps) + \frac{\eps}{b + \eps} (M,0)$, where $M = \frac{b+\eps}{\eps}(a-b)$. Using the concavity, $1$-homogeneity and normalisation of $\theta$, we obtain $\theta(a,b) \geq \frac{b}{b + \eps} \theta(b + \eps, b + \eps) = b \theta(1,1) = b$. The case $a < b$ can be treated analogously, and the case $a = b$ follows immediately by $1$-homogeneity and normalisation.

We refer to \cite[Lemma 2.2]{ErMa12} for a proof of \ref{it:hom-id}.
\end{proof}

It will be useful to associate a number $\lambda^{(\theta)}$ to an admissible mean $\theta$ that quantifies its asymmetry.

\begin{definition}[Weight]\label{def:weight}
For an admissible mean $\theta$, its \emph{weight} $\lambda^{(\theta)} \in [0,1]$ is given by
\begin{align*}
	\lambda^{(\theta)} = \partial_1 \theta(1, 1) \ .
\end{align*}
\end{definition}

If there is no danger of confusion, we will simply write $\lambda = \lambda^{(\theta)}$.

\begin{remark}\label{rem:weight}
Note that $\lambda^{(\theta)}$ is indeed nonnegative, since $\theta$ is non-decreasing in each of its arguments.
Applying \eqref{eq:4-pt} with $a = b = c = d = 1$, it follows that
\begin{align}\label{eq:ones}
	\partial_1 \theta(1,1) + \partial_2 \theta(1,1) = 1 \ ,
\end{align}
which implies that $\lambda^{(\theta)} \leq 1$.
\end{remark}

\begin{remark}\label{rem:weighted-inequalities}
It follows from \eqref{eq:4-pt} with $c = d = 1$, that
\begin{align*}
	\theta (a,b) \leq \lambda^{(\theta)}  a + (1 - \lambda^{(\theta)} ) b
\end{align*}
for any $a, b \geq 0$. Moreover, this inequality characterises  $\lambda^{(\theta)} \in [0,1]$ uniquely.
If $\theta$ is symmetric, it follows from \eqref{eq:ones} that $\lambda^{(\theta)} = \frac12$.
\end{remark}

Examples of symmetric admissible means are the arithmetic mean $\theta_{\mathrm{arith}}(a,b) = \frac{a+b}{2}$, the geometric mean $\theta_{\mathrm{geom}}(a,b) = \sqrt{a b}$, the logarithmic mean $\theta_{\mathrm{log}}(a,b) = \int_0^1 a^{1-p} b^p \dd p$, and the harmonic mean $\theta_{\mathrm{harm}}(a,b) = \frac{2ab}{a + b}$.
For each of these means there exist natural generalisations with weights $\lambda \in [0,1]$, such as
\begin{equation}\begin{aligned}\label{eq:weighted-means}
\theta_{\mathrm{arith}}^{(\lambda)}(a,b) & = \lambda a + (1-\lambda) b \ , &
\theta_{\mathrm{geom}}^{(\lambda)}(a,b) & =  a^\lambda b^{1-\lambda} \ , \\
\theta_{\mathrm{log}}^{(\lambda)}(a,b) & = \int_0^1 a^p b^{1-p} \, \tau(\ddd p) \ , &
\theta_{\mathrm{harm}}^{(\lambda)}(a,b) & = \frac{1}{\lambda / a + (1-\lambda) /  b} \ .
\end{aligned}\end{equation}
Here, $\tau$ is an arbitrary Borel probability measure on $[0,1]$ with $\int_0^1 p \tau(\ddd p) = \lambda$.

\begin{definition}[Mean and weight functions]\label{def:mean-function}
Let $\cX$ be a finite set.
\begin{enumerate}[label=(\roman*)]\setlength\itemsep{0ex}
\item A \emph{mean function} is a family of admissible means $\bftheta = (\theta_{xy})_{x,y \in \cX}$ satisfying the symmetry condition $\theta_{xy}(a,b) = \theta_{yx}(b,a)$ for all $x, y \in \cX$ and $a,b \geq 0$.
\item A \emph{weight function} is a collection $\bflambda = (\lambda_{xy})_{x,y \in \cX} \subseteq [0,1]$ satisfying $\lambda_{xy} + \lambda_{yx} = 1$ for all $x, y \in \cX$,
\item For a mean function $\bftheta$, its \emph{induced weight function} $\bflambda^{(\theta)}$ is defined by $	\lambda_{xy}^{(\theta)} = \partial_1 \theta_{xy}(1, 1)$ for all $x, y$.
We say that a mean function $\bftheta$ is \emph{compatible} with a weight function $\bflambda$, if $\bflambda$ is induced by $\bftheta$.
\end{enumerate}
\end{definition}

It follows from \eqref{eq:ones} that the induced weight function is indeed a weight function.

\medskip

We are now in a position to define the discrete transport metrics.
Given a Markov chain $(\cX, R, \pi)$ and a mean function $\bftheta$, the discrete transport metric is defined using discrete analogues of \eqref{eq:cont-action}.
The action functional $\cA : \cP(\cX) \times \R^\cX \to \R$ and its dual $\cA^* : \cP(\cX) \times \R^\cX \to \R \cup \{ +\infty\}$ are given by
\begin{equation}\begin{aligned}\label{eq:action-Markov}
\cA(\emm, \psi) & = \frac12 \sum_{x, y \in \cX}
		 \theta_{xy}\big(\emm(x) R(x,y), \emm(y) R(y,x) \big) \big( \psi(x) - \psi(y) \big)^2  \ ,
		 	\\
 \frac12 \cA^*(\emm, \sigma)
	& = \sup_{\psi \in \mathbb R^\cX} \bigg\{ \sum_{x \in \cX} \psi(x) \sigma(x) - \frac12 \cA(\emm, \psi) \bigg\} \ .
\end{aligned}\end{equation}
For $\emm_0, \emm_1 \in \cP(\cX)$, the associated transport metric is given by
\begin{align*}
 \cW(\emm_0,\emm_1)
   & =
  \inf \Bigg\{\sqrt{ \int_0^1 \mathcal A^*(\emm_t, \dot \emm_t) \dd t }\ : \
				(\emm_t)_{t \in [0,1]} \in \cAC(\emm_0, \emm_1) \Bigg\} \ .
\end{align*}
Here, $\cAC(\emm_0, \emm_1)$ denotes the class of all curves $(\emm_t)_{t\in[0,1]} \subseteq \cP(\cX)$ connecting $\emm_0$ and $\emm_1$, with the property  that $t\mapsto \emm_t(x)$ is absolutely continuous for all $x \in \cX$.

\begin{remark}\label{rem:general-mean}
In most of the previous papers dealing with discrete dynamical transport metrics, $\theta_{xy}$ has been chosen to be independent of $x$ and $y$. In particular, to obtain a gradient flow structure for Markov chains \cite{CHLZ11,Maas11,Miel11}, one chooses $\theta_{xy}$ to be the logarithmic mean for all $x,y$.
We will see in this paper that the additional freedom can be important to obtain Gromov--Hausdorff convergence.
In a noncommutative setting, a similar generalisation has been useful in the setting of Lindblad equations with a nontracial invariant state \cite{CaMa17,MiMi17}.
\end{remark}

We will occasionally use an equivalent formulation for $\cW$ given by
\begin{align*}
	\cW(\emm_0, \emm_1)
		= \inf\Bigg\{\sqrt{\int_0^1
			\cK(\emm_t, V_t)
		 \dd t} : (\emm, V) \in \cCE(\emm_0, \emm_1) \Bigg\} \ .
\end{align*}
Here, $\cCE(\emm_0, \emm_1)$ denotes the class of pairs $(\emm,V)$ satisfying
\begin{itemize}\setlength\itemsep{0ex}
\item $\emm :[0,T] \to \cP(\cX)$ is continuous with $\emm|_{t=0} = \emm_0$, and $\emm|_{t=1} = \emm_1$;
\item $V : [0,T] \times \cX \times \cX \to \R$ is locally integrable;
\item the \emph{continuity equation} holds in the sense of distributions:
    \begin{align}\label{eq:cont-eq}
      \dot \emm_t(x)
      + \frac12 \sum_{y \in \cX} \pi(x) R(x,y) \big( V_t(x,y) - V_t(y,x) \big)  = 0 \quad \quad \text{ for all } x \in \cX \ .
    \end{align}
\end{itemize}
Moreover,
\begin{align}\label{eq:convex-action}
	\cK(\emm, V)
	  := \frac12
			\sum_{x, y \in \cX} \pi(x) R(x,y)
			 K_{xy}\bigg(\frac{\emm(x)}{\pi(x)} , \frac{\emm(y)}{\pi(y)} , V(x,y) \bigg) \ ,
\end{align}
where the convex and lower semicontinuous function $K: \R^3 \to [0,\infty]$ is given by
\begin{align}\label{eq:def-K}
K_{xy}(a,b,w) :=
  \begin{cases}
    \frac{w^2}{\theta_{xy}(a,b)}\ , & w \in \R, \ a,b > 0\ ,\\
    0 \ , & w=0, \ a,b \geq 0 \ , \\
    +\infty\ , &\text{otherwise} \ .
  \end{cases}
\end{align}
By Fenchel duality, $\cA^*$ can be obtained from $\cK$ by minimising over all solutions to the continuity equation
\begin{equation}\begin{aligned}\label{eq: infsup}
	\cA^*(\emm, \sigma)
		= \inf_{V} \bigg\{ &
			\cK(\emm, V)  : \\ & \
			\sigma(x)
      + \frac12 \sum_{y \in \cX} \pi(x) R(x,y) \big( V(x,y) - V(y,x) \big)  = 0 \quad  \forall x \in \cX
					\bigg\} \ .
\end{aligned}\end{equation}
Here the infimum runs over all vector fields $V : \cX \times \cX \to \R$. Without loss of generality we may impose the anti-symmetry condition $V(x,y) = - V(y,x)$ for all $x, y \in \cX$; see \cite[Section 2]{ErMa12} for more details.

\subsection{Admissible and \texorpdfstring{$\zeta$}{z}-regular meshes}\label{sec:admissible}
 Following \cite[Section 3.1.2]{EyGaHe00}, we introduce the notion of an admissible mesh.

\begin{definition}[Mesh]
\label{def:mesh}
A \emph{mesh} of $\Omega$ is a pair $(\cT, \{x_K\}_{K \in \cT})$ where
\begin{itemize}
\setlength\itemsep{0ex}
\item  $\cT = \{K\}_{K\in \cT}$ is a finite partition (i.e., a pairwise disjoint covering) of $\bOmega$ into sets (called cells) with nonempty and convex interior.
\item  $\{x_K\}_{K\in \cT} \subseteq \bOmega$ is a family of distinct points with $x_K \in \overline K$ for every $K \in \cT$.
\end{itemize}
\end{definition}

Note that all interior cells $K$ are polytopes.
Throughout the paper we will use the following notation:

\begin{itemize}
\setlength\itemsep{0ex}
    \item[-] $|K|$ denotes the Lebesgue measure of a cell $K\in \cT$.
    \item[-] $(K|L) = \overline K\cap \overline L$ is the flat convex surface with $(d-1)$-dimensional Hausdorff measure $|(K|L)|$. Two cells $K,L \in \cT$ with $K \neq L$ are called \emph{nearest neighbours} if $|(K|L)| > 0$. In this case we write $K\sim L$. We write $K \simeq L$ if $K = L$ or $K \sim L$.
    \item[-] $d_{KL} = |x_K - x_L|$.
	\item[-] $[\cT] = \max_{K\in \cT} \diam(K)$ denotes the \emph{mesh size} of $\cT$.
\end{itemize}

\begin{definition}[Admissible and $\zeta$-regular mesh]
\label{def:admissible-mesh}
A mesh $(\cT, \{x_K\}_{K \in \cT})$ is called \emph{admissible} if $x_K - x_L \perp (K|L)$ whenever $K,L \in\cT$ are nearest neigbours.
An admissible mesh $\cT$ is called \emph{$\zeta$-regular} for $\zeta \in (0, 1]$ if the following  conditions hold:
\begin{itemize}\setlength\itemsep{0ex}
\item (inner ball condition) $B(x_K,\zeta[\cT])\subseteq K$ for every $K\in \cT$;
\item (area bound) $|(K|L)|\geq \zeta [\cT]^{d-1}$ for every $K, L \in \cT$ with $K \sim L$.
\end{itemize}

\end{definition}

In view of Definition \ref{def:admissible-mesh}, we stress that the $\zeta$-regularity of a mesh implies its admissibility.

Every Voronoi tesselation yields an admissible mesh. Another example is obtained by slicing $\bOmega$ multiple times in the cardinal directions. In this case there are several degrees of freedom for placing the points $\{x_K\}_{K\in \cT}$, so these points are not necessarily uniquely determined by $\cT$. We refer to \cite{EyGaHe00} for more information.

In the following result we collect some basic geometric properties of $\zeta$-regular meshes that will be useful in the sequel.

\begin{lemma}\label{lem:geom-ineq}
Let $\cT$ be a $\zeta$-regular mesh of $\Omega$ for some $\zeta \in (0,1]$. Then there exists a constant $C < \infty$ depending only on $\Omega$ and $\zeta$ such that the following assertions hold:
\begin{enumerate}[label=(\roman*)]\setlength\itemsep{0ex}
\item \label{it:nbh-bound}
The number of nearest neighbours of any cell is bounded by $C$.
 \item
 \label{it:chain-bound}
Any pair of cells $K, L \in \cT$ can be connected (for some $N \geq 0$) by a path $(K_i)_{i=1}^N \subseteq \cT$ with $K_0 = K$, $K_N = L$, $K_{i-1} \sim K_{i}$ for $1 \leq i \leq N$, and
\begin{align}\label{eq:good-path}
	\sum_{i=1}^N d_{K_{i-1},K_i} \leq C d_{KL} \ .
\end{align}
 \item
  \label{it:zeta-bounds}
For all $K \sim L$ we have the following estimates:
\begin{align}
	\label{eq:diam-bound1} \diam(K) & \leq C d_{KL} \ , \\
	\label{eq:cell-upperbound}  C^{-1} | K | & \leq d_{KL} | (K|L) |  \leq C | K | \ , \\
	\label{eq:cell-bound1} |K| & \leq C |L|   \ .
\end{align}
\end{enumerate}
\end{lemma}

\begin{proof}
{\ref{it:nbh-bound}}:
Write $h = [\cT]$ for brevity.
By the inner ball condition we have
\begin{align*}
\bigcup_{L: L \sim K} B(x_L, \zeta h)
	\subseteq \bigcup_{L: L \sim K} L
	\subseteq B(x_K, 2 h) \ .
\end{align*}
Hence, Assertion {\ref{it:nbh-bound}} follows by comparing volumes.

{\ref{it:chain-bound}}: 
Let $\ell$ be the line segment from $x_K$ to $x_L$.
By path-connectedness, there exists a continuous curve $\gamma$ in the $[\cT]$-neighbourhood of $\ell$ that connects $x_K$ to $x_L$ and avoids the boundaries of the cell interfaces.
Let $\{ K_i \}_{i = 0}^N$ be the $N$ cells that successively intersect 
$\gamma$.
By $\zeta$-regularity, each of the balls $B(x_{K_i}, \zeta [\cT])$ is contained in $K_i$. In turn, each of the cells $K_i$ is contained in the cylinder of radius $2[\cT]$, whose central axis is obtained by extending $\ell$ by a distance $[\cT]$ in both directions. 

The number of disjoint balls of radius $r$ that can be packed into a cylinder of length $s$ and radius $R$ is bounded by $C_d s \frac{R^{d-1}}{r^d}$, where $C_d$ is a dimensional constant. Therefore, if follows that 
\begin{align*}
	N \leq C \frac{ | x_K - x_L| + 2 [\cT] }{[\cT]}   \ , 
\end{align*} 
where $C$ depends on $\zeta$ and $d$. As $\zeta [\cT] \leq \min_{K' \neq L'} | x_{K'} - x_{L'} |$, if follows that $N \leq C | x_K - x_L| / [\cT]$. As $d_{K_{i-1}, K_i} \leq 2 [\cT]$, this yields the desired bound.

\ref{it:zeta-bounds}:
To prove \eqref{eq:diam-bound1}, we use the inner ball condition to obtain
$
	d_{KL}\geq \zeta h \geq \zeta \diam(K) .
$

Since $x_K - x_L \perp (K|L)$, the volume of the bipyramid spanned by $(K|L)$, $x_K$ and $x_L$ is given by $\frac1d|(K|L)| d_{KL}$.
Using the inner ball condition we obtain
\begin{align*}
	 \frac1d d_{KL} | (K|L) |
	 & \leq  |K \cup L|
	   \leq | B(x_K, 2 h) |
	   \leq C | B(x_K, \zeta h) |
	   \leq C | K | \ ,
\end{align*}
which proves the upper bound in \eqref{eq:cell-upperbound}. To prove the corresponding lower bound, note that $\zeta h \leq d_{KL}$, hence by $\zeta$-regularity,
\begin{align*}
	|K| \leq | B(x_K, h) |
		\leq C h^d
		\leq C h |(K|L)|
		\leq C d_{KL} |(K|L)| \ .
\end{align*}

The inequality \eqref{eq:cell-bound1} follows from \eqref{eq:cell-upperbound}.
\end{proof}

The $\zeta$-regularity condition allows us to control the constants in several useful inequalities. Most notably, we will use the Poincar\'e inequality \cite{PaWe60}
\begin{align}\label{eq:Poincare}
	\int_K \phi^2 \dd x \leq \frac{\diam(K)^2}{\pi^2} \int_K |\nabla \phi|^2 \dd x
\end{align}
and the trace inequality
\begin{align}\label{eq:trace-embedding}
	\int_{\partial K} \phi^2 \dd S \leq C \diam(K) \int_K |\nabla \phi|^2 \dd x \ ,
\end{align}
both of which are valid for all $K \in \cT$ and all $\phi \in H^1(K)$ with $\int_K \phi \dd x = 0$; cf. \cite[Theorem 1.5.1.10]{Gris11}. The constant $C < \infty$ depends only on $\zeta$ and the dimension $d$.
For the convenience of the reader we record a simple consequence that we will use below.

\begin{lemma}\label{lem:trace-consequence}
Let $\zeta \in (0,1]$.
There exists a constant $C < \infty$ depending on $\zeta$ and $\Omega$, such that, for any $K \in \cT$ and any $\phi \in H^1(K)$,
\begin{align}\label{eq:trace-consequence}
	\dashint_{\partial K} |\phi| \dd S
	 \leq   \dashint_{K} |\phi| \dd x
	 		+ C [\cT] \sqrt{\dashint_{K} |\nabla \phi|^2 \dd x }  \ .
\end{align}
Moreover, for any convex subset $B \subseteq K$ with $|B| \geq \zeta |K|$, we have
\begin{align}\label{eq:Poincare-consequence}
	\bigg| \dashint_{K} \phi \dd x  - \dashint_{B} \phi \dd x \bigg|
	\leq C [\cT] \sqrt{\dashint_{K} |\nabla \phi|^2 \dd x }  \ .
\end{align}
\end{lemma}

\begin{proof}
Write $\bar \phi = \dashint_K \phi \dd x$. Using \eqref{eq:trace-embedding} and \eqref{eq:cell-upperbound} we obtain
\begin{align*}
 \| \phi \|_{L^1(\partial K)}
 	& \leq  |\partial K| \, |\bar \phi| + \| \phi - \bar \phi \|_{L^1(\partial K)}
 	\\ & \leq  |\partial K| \, |\bar \phi| + C \sqrt{|\partial K| \diam(K)} \|\nabla \phi\|_{L^2(K)} \ .
\end{align*}
As Lemma \ref{lem:geom-ineq} implies $|K| \leq C |\partial K|\, [\cT]$, we have $\diam(K) \leq C \frac{|\partial K|}{|K|}[\cT]^2$, which yields \eqref{eq:trace-consequence}.

To prove \eqref{eq:Poincare-consequence}, we write $\bphi_K = \dashint_{K} \phi \dd x$ and $\bphi_B  = \dashint_{B} \phi \dd x$ for brevity. Using the Poincar\'e inequality \eqref{eq:Poincare} we obtain
\begin{align*}
|K| \, |\bphi_B - \bphi_K|^2
	& \leq C |B| \, |\bphi_B - \bphi_K|^2
	\\& \leq C \int_B |\bphi_B - \phi|^2 \dd x + C \int_K |\phi - \bphi_K|^2 \dd x
	\\& \leq C [\cT]^2 \int_K |\nabla \phi|^2 \dd x \ ,
\end{align*}
which implies \eqref{eq:Poincare-consequence}.
\end{proof}

\subsection{Discrete optimal transport on admissible meshes}\label{sec:admissible2}

Given an admissible mesh $\cT$ of $\Omega$ we consider an irreducible Markov chain on $\cT$ with transition rates
\begin{align}\label{eq:rates}
	 R(K,L) =
	  	\frac{| (K|L) |}{|K|d_{KL}} \quad \text{ if $K \sim L$ }
\end{align}
and $R(K,L) = 0$ otherwise.

\begin{remark}[Formal consistency]\label{rem:consistency}
This choice of the transition rates $R(K,L)$ is motivated by the following formal consistency computation for the Dirichlet energy associated to our problem.
Let $\cU$ be the uniform probability measure on $\Omega$. For a  smooth function $\phi : \Omega \to \R$, the continuous action functional satisfies
\begin{align*}
	&\bA(\cU, \phi)
		 = \frac{1}{|\Omega|}\int_\Omega |\nabla \phi|^2 \dd x
		= - \frac{1}{|\Omega|} \sum_K \int_K \phi \Delta \phi \dd x
		\approx - \frac{1}{|\Omega|} \sum_K \phi(x_K)
			\int_{\partial K} \partial_n \phi \dd S
		\\&  \approx - \frac{1}{|\Omega|} \sum_{K,L} \phi(x_K)
			|(K|L)| \frac{\phi(x_L) - \phi(x_K)}{|x_K - x_L|}
		 =  \frac{1}{2|\Omega|} \sum_{K,L} 					 \frac{|(K|L)|}{|x_K - x_L|} \big(\phi(x_K) - \phi(x_L)\big)^2
 \ .
\end{align*}
Let $U_\cT \in \cP(\cT)$ be the canonical discretisation of $\cU$ given by $U_\cT = P_\cT \cU$, and let $\psi_\cT : \cT \to \R$ be given by $\psi_\cT(K) = \phi(x_K)$.
Then the latter expression is indeed of the form $\cA(U_\cT, \psi_\cT)$ defined in \eqref{eq:action-Markov}, provided that the coefficients $R(K,L)$ are defined by \eqref{eq:rates}.
\end{remark}

It is immediate to check that the unique invariant probability measure $\pi$ on $\cT$ is given $\pi(K) = \frac{|K|}{|\Omega|}$. Moreover, the detailed balance condition $\pi(K) R(K, L) = \pi(L) R(L, K)$ holds.

For this Markov chain, the discrete action functional $\cA_\cT : \cP(\cT) \times \R^\cT \to \R$ and its dual $\cA_\cT^* : \cP(\cT) \times \mathbb R^\cT \to \R \cup \{ +\infty\}$ defined in \eqref{eq:action-Markov} are given by
\begin{align*}
\cA_\cT(\emm, \psi) & = \frac12 \sum_{K,L \in \cT} \frac{|(K|L)|}{d_{KL}}
		 \theta_{KL}\bigg(\frac{\emm(K)}{|K|}, \frac{\emm(L)}{|L|}\bigg) \big( \psi(K) - \psi(L) \big)^2 \ ,
		 	\\
 \frac12 \cA_\cT^*(\emm, \sigma)
	& = \sup_{\psi\in \mathbb R^\cT} \bigg(\sum_{K \in \cT} \psi(K) \sigma(K) - \frac12 \cA_\cT(\emm, \psi) \bigg) \ .
\end{align*}
The main object of study in this paper is the associated transport metric, defined for $\emm_0, \emm_1 \in \cP(\cT)$ by
\begin{align*}
 \cW_\cT(\emm_0, \emm_1)
 & =
					 \inf \bigg\{\sqrt{ \int_0^1 \mathcal A_\cT^*(\emm_t, \dot \emm_t) \dd t} \ : \
				(\emm_t)_{t \in [0,1]} \in \cAC(\emm_0,\emm_1) \bigg\} \ .
\end{align*}

\section{A priori estimates}\label{sec:a-priori}

In this section we prove the necessary \emph{a priori} estimates.
Throughout this section, we fix a convex bounded open set $\Omega \subseteq \R^d$ and a $\zeta$-regular mesh $\cT$ for some $\zeta \in (0,1]$.
Moreover, we use the convention that the constants $C$ appearing in this section (which are allowed to change from line to line) may depend on $\Omega$ and $\zeta$, but not on other properties of $\cT$.

\begin{lemma}
\label{lem:neigbour-bound}
There exists a constant $C < \infty$ such that for any $K, L \in \cT$ with $K \sim L$,
\begin{align*}
	\cW_\cT( \delta_K, \delta_L) \leq C  d_{KL} \ .
\end{align*}
\end{lemma}

\begin{proof}
It follows from \cite[Lemma 2.13]{ErMa12} that
\begin{align*}
 \cW_\cT( \delta_K, \delta_L )^2 \leq c \frac{d_{KL}}{|(K|L)|} \max\{ |K| , |L| \}
\end{align*}
for some universal constant $c < \infty$. The claim follows by $\zeta$-regularity in view of \eqref{eq:cell-upperbound}.
\end{proof}

To compare discrete and continuous measures, we use the canonical projection operator $P_\cT$ defined in \eqref{eq:proj-embed}. The associated embedding operator $\cQ_\cT : \cP(\cT) \to \cP(\bOmega)$ is given by
\begin{align}\label{eq:embedding}
	\cQ_\cT \emm  & = \sum_{K \in \cT} \emm(K) \cU_K \qquad \text{for } \emm  \in \cP(\cT) \ ,
\end{align}
where $\cU_K$ denotes the uniform probability measure on $K$.
Note that $P_\cT\circ \cQ_\cT$ is the identity on $\cP(\cT)$.
The following lemma quantifies how close $\cQ_\cT\circ P_\cT$ is to the identity on $\cP(\bOmega)$.

\begin{lemma}[Consistency]\label{lem:almost-identity}
For all $\mu\in \cP(\bOmega)$ we have
\begin{align*}
	\bW(\mu, \cQ_\cT P_\cT \mu)\leq [\cT] \ .
\end{align*}
\end{lemma}

\begin{proof}
For $U \in \cT$, let $\gamma_K \in \cP(\bOmega \times \bOmega)$ be any coupling between $\tilde\mu_K$, the normalised restriction of $\mu$ to $K$, and $\cU_K$, the uniform probability measure on $K$. It then follows that $\gamma := \sum_{K \in \cT} \mu(K) \gamma_K$ belongs to $\Gamma(\mu, \cQ_\cT P_\cT \mu)$. Consequently,
\begin{align*}
	\bW(\mu, \cQ_\cT P_\cT \mu)^2
		& \leq \int_{\bOmega \times \bOmega} |x-y|^2 \dd \gamma(x,y)
		\\& = \sum_{K} \mu(K) \int_{K \times K} |x-y|^2 \dd \gamma_K(x,y)
	    \leq \sum_{K} \mu(K) \diam(K)^2
		\leq [\cT]^2 \ .
\end{align*}
\end{proof}

The following result provides a coarse bound for the discrete distance $\cW_\cT$ in terms of the Kantorovich distance $\bW$.

\begin{lemma}[Upper bound for $\cW_\cT$]\label{lem:distconttodisc}
For all $\emm_0, \emm_1\in \cP(\cT)$ we have
\begin{align}
    \cW_\cT(\emm_0, \emm_1)
    	\leq C \Big( \bW(\cQ_\cT \emm_0, \cQ_\cT \emm_1) + [\cT] \Big) \ .
    \end{align}
\end{lemma}

\begin{proof}
Let $\gamma \in \Gamma_o(\cQ \emm_0, \cQ \emm_1)$ be an optimal plan for $\bW$, and set $\gamma_{KL} := \gamma(K \times L)$ for brevity.
Observe that $\emm_0 = \sum_{K,L} \gamma_{KL} \delta_K$ and $\emm_1 = \sum_{K,L} \gamma_{KL} \delta_L$.
Therefore, by convexity of the squared distance (cf. \cite[Proposition 2.11]{ErMa12}) we have
\begin{align*}
	\cW_\cT(\emm_0, \emm_1)^2
		\leq
		\sum_{K,L} \gamma_{KL} \cW_\cT( \delta_K , \delta_L)^2 \ .
\end{align*}
For $K, L \in \cT$, take a connecting path $\{K_i\}_{i=0}^N \subseteq \cT$ with $K_0 = K$ and $K_N = L$ satisfying the $\zeta$-regularity estimate \eqref{eq:good-path}. Using Lemma \ref{lem:neigbour-bound} we obtain
\begin{align*}
	\cW_\cT( \delta_K, \delta_L )
	\leq \sum_{i=1}^N	\cW_\cT( \delta_{K_{i-1}}, \delta_{K_{i}} )
	\leq C \sum_{i=1}^N d_{K_{i-1},K_i}
	\leq C d_{KL} \ .
\end{align*}
For all $x \in K$ and $y \in L$ we have $d_{KL}^2 = |x_{K} - x_{L}|^2 \leq |x - y|^2 + C [\cT]$. Combining these estimates, the result follows.
\end{proof}

The following lemma provides a coarse lower bound for the discrete dual action functional in terms of its continuous counterpart.

\begin{lemma}[Bound for the dual action functional] \label{lem:a-priori}
There exists a constant $C < \infty$ such that for any $\emm \in \cP(\cT)$, and any $\sigma \in \cM_0(\cT)$,
we have
\begin{align*}
	\bA^*(H_{[\cT]} \cQ_\cT \emm, H_{[\cT]} \cQ_\cT \sigma)
		\leq C \cA_\cT^*(\emm, \sigma) \ .
\end{align*}
\end{lemma}

\begin{proof}
Let $m\in \cP(\cT)$ and $\sigma \in \cM_0(\cT)$ be such that $\cA_\cT^*(\emm, \sigma) < \infty$. 
In view of \eqref{eq: infsup} there exists an anti-symmetric momentum vector field $V: \cT\times \cT \to \R$ solving the discrete continuity equation
\begin{align}\label{eq:discrete-cont}
\sigma(K) + \sum_L \frac{|(K|L)|} {d_{KL}} V({K,L})  = 0 \ .
\end{align}
For $K \in \cT$, define $g : \partial K \to \R$ by $g(x) = \frac{V(K,L)}{d_{KL}}$ for $x \in (K|L)$ with $K \sim L$, and $g(x) = 0$ for $x \in \partial \Omega$.
It then follows that $\sigma(K) = - \int_{\partial K} g \dd S$.
Therefore, denoting the outward unit normal on $\partial K$ by $\bfn$, the Neumann problem
 \begin{equation}\label{eq:Neumann-problem}
  \begin{cases}
  - \Delta \phi_K = \sigma(K)/ |K| & \text{in } K\\
   \partial_\bfn \phi_K  = g &\text{on } \partial K
  \end{cases}
 \end{equation}
has a unique solution $\phi_K \in H^1(K)$ with $\int_K \phi_K \dd x =0$.
Since $\sigma(K) \leq \|g\|_{L^1(\partial K)}$ by \eqref{eq:discrete-cont}, we obtain
\begin{equation}\begin{aligned}
\label{eq: Poincare estimate on cells}
 \int_K |\nabla \phi_K|^2 \dd x
  &= \frac{1}{|K|} \int_K \sigma(K) \phi_K \dd x + \int_{\partial K}  g \phi_K \dd S
  \\& \leq  \frac{|\sigma(K)|}{|K|}   \, \|\phi_K\|_{L^1(K)}
    +  \| g \|_{L^2(\partial K)} \| \phi_K\|_{L^2(\partial K)}
\\& \leq  \|g\|_{L^2(\partial K)}\left( \sqrt \frac{|\partial K|}{|K|} \|\phi_K\|_{L^2(K)} + \| \phi_K\|_{L^2(\partial K)}  \right) \ .
\end{aligned}\end{equation}
Since $\int_K \phi_K \dd x = 0$, the right-hand side can be bounded in terms of $\|\nabla \phi_K\|_{L^2(K)}$ using the Poincar\'e inequality \eqref{eq:Poincare} and the trace inequality \eqref{eq:trace-embedding}. Moreover, using the $\zeta$-regularity inequalities \eqref{eq:diam-bound1} \eqref{eq:cell-upperbound}, and the bound on the number of neighbours from Lemma \ref{lem:geom-ineq}, we obtain $|\partial K| \diam(K) \leq C |K|$.
Consequently,
\begin{align*}
	 \int_K |\nabla \phi_K|^2 \dd x
	  &\leq C \diam(K) \int_{\partial K} g^2  \dd S \ .
\end{align*}
Using the $\zeta$-regularity inequality \eqref{eq:diam-bound1} we obtain,
\begin{align*}
	\int_{\partial K} g^2  \dd S
	= \sum_{L\in \cT} \frac{|(K|L)|}{d_{KL}^2} V({K,L})^2
	\leq \frac{C}{\diam(K)}\sum_{L\in \cT} \frac{|(K|L)|}{d_{KL}}  V({K,L})^2 \ ,
\end{align*}
and therefore,
\begin{align*}
	\int_K |\nabla \phi_K|^2\dd x
	\leq C \sum_{L \in \cT}  \frac{|(K|L)|}{d_{KL}}  V({K,L})^2 \ .
\end{align*}

Let us now define the vector field $\Phi \in L^2(\bOmega;\R^d)$ by $\Phi(x) = \nabla \phi_K(x)$ for $x\in K$.
By \eqref{eq:Neumann-problem} and the anti-symmetry of the Neumann boundary values, one has
 \begin{equation}\label{eq:cont-rhodot-V}
Q_\cT \sigma + \dive \Phi = 0
 \end{equation}
in the sense of distributions, and $\Phi \cdot n = 0$ on $\partial \Omega$.

Take now $\emm \in \cP(\cT)$. Writing $\rho(K) = \frac{\emm(K)}{|K|}$, we obtain
\begin{equation}\begin{aligned}
\label{eq:action-bound}
	 \bA^*(\cQ_\cT \emm, \cQ_\cT \sigma)
 		& \leq \sum_{K\in \cT} \int_K\frac{ |\nabla \phi_K|^2}{\rho(K)} \dd x
		\\ &
		 \leq C \sum_{K,L\in \cT}
 \frac{|(K|L)|}{d_{KL}} \frac{V({K,L})^2}{\theta_{\textrm{harm}}(\rho(K),\rho(L))} \  ,
\end{aligned}\end{equation}
where $\theta_{\textrm{harm}}(a,b) = \frac{2ab}{a + b}$ is the harmonic mean.

We would like to obtain a similar estimate involving the means $\theta_{KL}$, but as $\rho$ is in general not bounded away from $0$, we cannot bound the harmonic mean by a multiple of $\theta_{KL}$. To remedy this issue, we perform an additional regularisation step.
Consider the function $\tilde\rho : \cT \to \R_+$ given by $\tilde\rho(K) = \rho(K) + \sum_{L\sim K}\rho(L)$, and set $\tilde\emm(K) = \tilde\rho(K)|K|$.
In view of \eqref{eq:action-monotone}, Lemma \ref{lem:symmetrized} below, and \eqref{eq:action-heat}, we obtain, for $a= [\cT]$,
\begin{align*}
		\bA^*(H_a Q_\cT \emm, H_a Q_\cT \sigma)
	& \leq C
		 \bA^*(H_a Q_\cT \tilde \emm, H_a Q_\cT \sigma)
	\\& \leq C
		 \bA^*(Q_\cT \tilde \emm, Q_\cT \sigma )
\end{align*}
We stress that to obtain the first inequality, the choice $a= [\cT]$ is crucial; cf. Remark \ref{rem:time-choice}.
Moreover, applying \eqref{eq:action-bound} to $\tilde\emm$, we obtain
\begin{align*}
	\bA^*( Q_\cT \tilde\emm , Q_\cT \sigma )
	\leq C \sum_{K,L\in \cT}
 \frac{|(K|L)|}{d_{KL}} \frac{V({K,L})^2}{\theta_{\textrm{harm}}(\tilde\rho(K),\tilde\rho(L))} \ .
\end{align*}
Since $\tilde \rho(K), \tilde \rho(L) \geq \rho(K) + \rho(L)$ whenever $K \sim L$, we have $\theta_{\textrm{harm}}(\tilde \rho(K), \tilde \rho(L)) \geq \rho(K) + \rho(L) \geq \theta_{KL}(\rho(K), \rho(L))$.
Therefore, combining the estimates above, we obtain
\begin{align*}
	\bA^*(H_a Q_\cT \emm, H_a Q_\cT \sigma)
		\leq C  \sum_{K,L\in \cT}
 \frac{|(K|L)|}{d_{KL}} \frac{V({K,L})^2}{\theta_{KL}(\rho(K), \rho(L))}
 \ .
\end{align*}
Taking the infimum over all $V$ satisfying \eqref{eq:discrete-cont}, the result follows using \eqref{eq: infsup}.
\end{proof}

The following result was used in the proof of Lemma \ref{lem:a-priori}.
Recall that we write $K \simeq L$ if $K\sim L$ or $K=L$.

\begin{lemma}\label{lem:symmetrized}
For $\emm \in \P(\cT)$ we define $\tilde \emm : K \to \R_+$ by $\tilde \emm(K) = |K| \tilde\rho(K)$, where $\tilde \rho(K) = \sum_{L \simeq K} \rho(L)$ and $\rho(K) = \frac{\emm(K)}{|K|}$.
Then, for every $a > 0$ and $x \in \bOmega$, the inequality
\begin{align}
 H_a Q_\cT \emm (x) \leq H_a Q_\cT \tilde \emm(x) \leq C e^{C \sqrt{\eta(a)} [\cT]} H_a Q_\cT \emm (x) \ ,
\end{align}
holds, with $\eta(a) = a^{-2} + (a^{-1}\log a)\vee 0$.
 \end{lemma}

\begin{remark}\label{rem:time-choice}
It is crucial in our application that by choosing $a = [\cT]$ (as is done in Lemma \ref{lem:a-priori}), the exponent $\sqrt{\eta(a)} [\cT]$ remains bounded as $[\cT] \to 0$.
\end{remark}

\begin{proof}
Since $\tilde \emm \geq \emm$, the first inequality follows from the positivity of $H_a$, so it remains to prove the second inequality.

To prove the second inequality, note that
\begin{align*}
	H_a Q_\cT \emm(x)
	 & = \sum_K \rho(K) \int_K h_a(x,y) \dd y
	  \quad \ \text{and} \ \, \\
	H_a Q_\cT \tilde \emm(x)
	 & = \sum_K \rho(K) \sum_{L : L \simeq K} \int_{L} h_a(x,y) \dd y \ .
\end{align*}
We claim that there exists a constant $C < \infty$ such that
\begin{align}\label{eq: quotient estimate}
    h_a(x,y) \leq e^{C \sqrt{\eta(a)} |y - z| } h_a(x,z)
\end{align}
for $x, y, z \in \bOmega$.

Let $T_{KL}$ be a (not necessarily optimal) transport map between the uniform probability measures on neighbouring cells $K$ and $L$. As $|T_{KL}(z) - z| \leq 2 [\cT]$ for $z \in K$,  the claim yields
\begin{align*}
	\frac{1}{|L|}\int_{L} h_a(x,y) \dd y
	= \frac{1}{|K|}\int_{K} h_a(x,T_{KL}(z)) \dd z
	\leq e^{C \sqrt{\eta(a)}[\cT]} \frac{1}{|K|}\int_{K} h_a(x,z) \dd z \ .
\end{align*}
Since $|L| \leq C |K|$ by the $\zeta$-regularity estimate \eqref{eq:cell-bound1}, and since $\# \{L : L \simeq K\} \leq C$ by Lemma \ref{lem:geom-ineq}, we obtain
\begin{align*}
	\sum_{L : L \simeq K} \int_{L} h_a(x,y) \dd y
	\leq C e^{C \sqrt{\eta(a)}[\cT]}
			\int_{K} h_a(x,z) \dd z \ ,
\end{align*}
which yields the result.

It remains to prove the claim \eqref{eq: quotient estimate}.
For this purpose, note that by the heat kernel bounds \eqref{eq: heat kernel estimates} there exist $\bOmega$-dependent constants $c, C> 0$ such that, for all $a > 0$,
\begin{align*}
\sup_{x,y \in \Omega} h_{a/2} (x,y) \leq C (1 \vee a^{-d/2})
    \quad \text{and}\quad
    \inf_{x,y \in \Omega} h_a (x,y) \geq C a^{-d/2} e^{-c/a} \ .
\end{align*}
For any smooth function $\phi : \Omega \to (0,M]$ with $M \in (0,\infty)$, the local logarithmic Sobolev inequality \eqref{eq:LSI-local} implies that
\begin{align*}
	  |\nabla \log H_{a/2} \phi|^2
 	\leq \frac{2}{a}\log \bigg(\frac{M}{H_{a/2} \phi} \bigg) \ .
\end{align*}
Applying this inequality with $\phi = h_{a/2}(x, \cdot)$, we obtain using the semigroup property,
    \begin{align}
      \sup_{x,y\in \Omega} |\nabla_y \log h_a(x,y)|^2
      \leq \frac{2}{a}
      \log\left( \frac{\sup_{x,y\in \Omega} h_{a/2}(x,y)}{\inf_{x,y\in \Omega}
      h_a(x,y)} \right) \leq
      C \eta(a)  \ ,
    \end{align}
which implies \eqref{eq: quotient estimate}.
\end{proof}

\section{Finite volume estimates for discrete optimal transport}\label{sec:finite-volume}

The goal of this section is to show that the dual action functional $\cA_\cT^*$ for the discrete transport problems is a good approximation to its continuous counterpart $\bA^*$. This will be shown in Proposition \ref{prop:cont-and-disc-action}.

To obtain this result, we first show an error estimate for a discrete elliptic problem, in the spirit of \cite[Theorem 3.5]{EyGaHe00}.
In our application, we think of $w$ as being $\dot \mu_t$ at some fixed time $t$, so that the elliptic equation below is the continuity equation at time $t$.

As in Section \ref{sec:a-priori}, we fix a convex bounded open set $\Omega \subseteq \R^d$ and a $\zeta$-regular mesh $\cT$ for some $\zeta \in (0,1]$.

\begin{proposition}[Weighted $H^1$-error bound]\label{prop:error-est}
Let $w \in L^2(\Omega)$ with $\int_\Omega w(x) \dd x=0$ and $\mu = u \dd x \in \cP_\delta(\bOmega)$ be given, and
let $\phi \in H^2(\Omega)$ be the unique variational solution to
	\begin{equation}
		\begin{cases}\label{eq:elliptic}
		- \dive (u \nabla \phi)  = w &\text{ in }\Omega\\
		\partial_\bfn \phi  = 0	&\text{ on }\partial\Omega
		\end{cases}
	\end{equation}
   satisfying $\int_\Omega \phi \dd x = 0$.

Define $\emm \in \cP(\cT)$ by $\emm = P_\cT \mu$, and $\sigma \in \cM_0(\cT)$ by $\sigma := P_\cT w$.
We write $\rho(K) := \emm(K)/|K|$ and $\hrho(K,L) := \theta_{KL}(\rho(K) , \rho(L))$.
Let $\psi : \cT \to \R$ be the unique solution to the corresponding discrete elliptic problem
	\begin{align}\label{eq:elliptic-discrete}
	- \sum_{L \in \cT} \frac{|(K|L)|}{d_{KL}} \hrho(K,L) \big(\psi(L) - \psi(K)\big)
		& = \sigma(K)
	\end{align}
satisfying $\sum_{K \in \cT} |K| \psi(K) = 0$.

Set $\bphi(K) = \dashint_{B_K} \phi \dd x$, where $B_K = B(x_K,\zeta [\cT])$ denotes the closed ball of radius $\zeta [\cT]$ around $x_K$, and set
\begin{align*}
	e(K) := \bphi(K) - \psi(K) \ .
\end{align*}
Then there exists a constant $C < \infty$ depending only on
	$\delta$, $\Omega$, and $\zeta$, such that
	\begin{align}\label{eq:errorest}
		\cA_\cT(\emm, e) \leq C  [\cT]^2 \, \|w\|^2_{L^2(\Omega)}\ .
	\end{align}
\end{proposition}

\begin{remark} \label{rem:existence}
The existence of a unique variational solution $\phi$ to the Neumann problem \eqref{eq:elliptic} in $H^1(\Omega)$ follows from the Lax--Milgram theorem.
The existence of a unique solution to \eqref{eq:elliptic-discrete} follows from elementary linear algebra, cf. \cite{Maas11}. In both cases, uniqueness holds up to an additive constant.
\end{remark}

\begin{remark}\label{rem:error-est}
Crucial for the proof is the \emph{a priori} estimate $\| \phi \|_{H^2(\Omega)} \leq C \|w\|_{L^2(\Omega)}$ with $C < \infty$ depending only on $\Omega$ and $\delta$; see \cite[Theorem 3.1.2.3]{Gris11}.
\end{remark}

\begin{remark}
The error estimate \eqref{eq:errorest} is similar to the $H^2$-error estimate in
\cite{EyGaHe00}, except that we use an averaged sample $\dashint_{B_K} \phi$ instead of the pointwise sample $\phi(x_K)$ to define the
error term. This change is required to be able to deal with dimensions $d \geq 4$, where $H^2(\Omega)$ does not embed into the space of continuous functions.
\end{remark}

\begin{proof}[Proof of Proposition \ref{prop:error-est}]
Integration of \eqref{eq:elliptic} over $K \in \cT$ yields
	\begin{align}\label{eq:integrated}
		- \sum_{L: L \sim K} \int_{(K|L)} u \partial_\bfn \phi  \dd S
		  = \int_K w \dd x
		  = \sigma(K) \ ,
	\end{align}
where $\bfn$ denotes the outward unit normal on $\partial K$.
We define
	\begin{align*}
	R(K,L) :=  \frac{1}{\hrho(K,L)} \dashint_{(K|L)} u \partial_\bfn \phi \dd S
			 - \frac{\bphi(L) -\bphi(K)}{d_{KL}} \ ,
	\end{align*}
and note that, by \eqref{eq:elliptic-discrete} and \eqref{eq:integrated},
\begin{align*}
		&\sum_{L: L \sim K} \frac{|(K|L)|}{d_{KL}} \hrho(K,L)  \big(e(K) - e(L)\big)
		 \\ & = \sum_{L: L \sim K} \bigg[ |(K|L)| \hrho(K,L)
		 		 \bigg( R(K,L) - \frac{1}{\hrho(K,L)}\dashint_{(K|L)} u \partial_{\bfn} \phi \dd S \bigg)\bigg]
		   		- \sigma(K)
		 \\& = \sum_{L: L \sim K} | (K|L) | \hrho(K,L)  R(K,L) \ .
\end{align*}
Multiplying this expression by $e(K)$, and using the symmetry of $\hrho(K,L)$ and the anti-symmetry of $R(K,L)$, we obtain
\begin{align*}
	\cA_\cT(\emm, e)
	       & = \frac12 \sum_{K,L} \frac{|(K|L)|}{d_{KL}} \hrho(K,L)  \big(e(K) - e(L)\big)^2
		 \\& =  \sum_K e(K)  \sum_{L: L \sim K} \frac{|(K|L)|}{d_{KL}} \hrho(K,L)   \big(e(K) - e(L)\big)
		 \\& = \sum_K e(K)  \sum_{L: L \sim K} | (K|L) | \hrho(K,L)  R(K,L)
		 \\& = \frac12 \sum_{K,L}  | (K|L) | \hrho(K,L) R(K,L)  \big(e(K) - e(L)\big)
		 \\& \leq
		 	\sqrt{ \frac12 \cA_\cT(\emm, e) \sum_{K,L}  d_{KL} |(K|L)| \hrho(K,L) R(K,L)^2}\ .
\end{align*}
Consequently,
	\begin{align}\label{eq:act-bound}
	\cA_\cT(\emm, e)
		&	\leq \frac12 \sum_{K,L} d_{KL} |(K|L)| \hrho(K,L) R(K,L)^2 \ .
	\end{align}
Observe that
\begin{align}\label{eq:two-integrals}
|R(K,L)|
	& \leq \dashint_{(K|L)} \left| \frac{u}{\hrho(K,L)} - 1 \right|
		|\partial_\bfn \phi| \dd S
	+ \left| \dashint_{(K|L)} \partial_\bfn \phi \dd S
		- \frac{\bphi(L) - \bphi(K)}{d_{KL}} \right| \ .
\end{align}

To estimate the first term on the right-hand side of \eqref{eq:two-integrals}, we note that the function $u$ is Lipschitz since $\mu \in \cP_\delta(\bOmega)$, and the mean $\theta_{KL}$ is Lipschitz on $[\delta,\infty)^2$.
Using this observation followed by Lemma \ref{lem:trace-consequence}, we obtain
\begin{equation} \label{eq: first integral}
\begin{aligned}
 \dashint_{(K|L)} \left| \frac{u}{\hrho(K,L)} - 1  \right|  |\partial_\bfn \phi| \dd S
	& \leq C [\cT] \dashint_{(K|L)} |\partial_\bfn \phi| \dd S
\\ & \leq C  [\cT] \bigg(
		\dashint_{K} |\nabla \phi| \dd x
		 + 	[\cT] \sqrt{\dashint_{K} |D^2 \phi|^2 \dd x }
				 \bigg) \ ,
\end{aligned}
\end{equation}
for some constant $C < \infty$ depending on $\Omega$, $\delta$ and $\zeta$.

To estimate the second integral in \eqref{eq:two-integrals} for neigbouring cells $K \sim L$, write $x_{KL} = x_L - x_K$, so that $n = \frac{x_{KL}}{d_{KL}}$. We set
\begin{align*}
	M_K := \dashint_K \nabla \phi \cdot n \ , \qquad
	M_{KL} := \dashint_{K \cup L} \nabla \phi \cdot n \ .
\end{align*}
Using the trace inequality \eqref{eq:trace-embedding} we have
\begin{align}\label{eq:no1}
	\bigg| \, \dashint_{(K|L)} \partial_\bfn \phi \dd S - M_K \bigg|^2
		 \leq
		 	 \dashint_{(K|L)} (\partial_\bfn \phi - M_K )^2  \dd S
		 \leq C [\cT] \frac{|K|}{| \partial K|}\dashint_K |D^2 \phi|^2 \dd x \ .
\end{align}
Arguing as in \eqref{eq:Poincare-consequence}, we obtain
\begin{align}\label{eq:no2}
	| M_K - M_{KL} | \leq C[\cT]
	 \sqrt{\dashint_{K \cup L} | D^2 \phi |^2 \dd x }   \ .
\end{align}
Furthermore, writing $B_K = B(x_K,\zeta [\cT])$ as before, the fundamental theorem of calculus yields
\begin{align*}
    \frac{\bphi(L)-\bphi(K)}{d_{KL}}
    & = \dashint_{B_K} \frac{\phi(y + x_{KL}) - \phi(y)}{d_{KL}}  \dd y
    \\& = \dashint_{B_K} \int_0^1 \partial_n \phi(y + t x_{KL}) \dd t \dd y
    = \dashint_{K \cup L} \partial_n \phi(x) f(x) \dd x \ ,
\end{align*}
where $f$ is a nonnegative function satisfying $\dashint_{B_K} f(x) \dd x = 1$ and $\|f\|_{L^\infty} \leq C < \infty$.
Therefore, using the Poincar\'e inequality \eqref{eq:Poincare},
\begin{equation}\begin{aligned}\label{eq:no3}
	\bigg| \frac{\bphi(L)-\bphi(K)}{d_{KL}} - M_{KL} \bigg|
	& \leq \dashint_{K \cup L} | \partial_n \phi(x) - M_{KL} | f(x) \dd x
	\\& \leq \|f\|_{L^\infty}  \sqrt{\dashint_{K \cup L} | \partial_n \phi(x) - M_{KL} |^2 \dd x }
	\\& \leq C [\cT]  \sqrt{\dashint_{K \cup L} | D^2 \phi |^2 \dd x } \ .
\end{aligned}\end{equation}
Combining the inequalities \eqref{eq:no1}, \eqref{eq:no2} and \eqref{eq:no3}, we obtain, using the $\zeta$-regularity once more,
\begin{align*}
	\bigg| \, \dashint_{(K|L)} \partial_\bfn \phi \dd S
		 - \frac{\bphi(L) - \bphi(K)}{d_{KL}} \bigg|
\leq C [\cT] \sqrt{\dashint_{K \cup L} | D^2 \phi |^2 \dd x } \ .
\end{align*}
Together with \eqref{eq: first integral}, the latter estimate yields
\begin{align*}
|R(K,L)| \leq C [\cT]
			\sqrt{\dashint_{K\cup L} |\nabla \phi|^2 + |D^2 \phi|^2 \dd x } \ .
\end{align*}
Thus, using \eqref{eq:act-bound} we find
\begin{equation}
\begin{aligned}
	\cA_\cT(\emm, e)
	   & \leq \frac12\sum_{K,L} d_{KL} |(K|L)| \hrho(K,L) R(K,L)^2
	\\ & \leq C [\cT]^2 \sum_{K,L}
		\frac{d_{KL} |(K|L)|}{|K \cup L|}
			\int_{K\cup L} |\nabla \phi|^2 + |D^2 \phi|^2 \dd x
	\\ & \leq C [\cT]^2 \|\phi\|_{H^2(\Omega)}^2 \ ,
\end{aligned}
\end{equation}
since $d_{KL}(K|L)/|K\cup L| \leq C$, and the maximum number of neigbours of any cell is bounded in view of Lemma \ref{lem:geom-ineq}. The result thus follows using the \emph{a priori} bound from Remark \ref{rem:error-est}.
\end{proof}

In the following result it suffices to assume that $\cT$ is admissible. We do not need to require $\zeta$-regularity.

\begin{proposition}[Discrete weighted Poincar{\'e} inequality]\label{prop:Poincare-discrete}
There exists a constant $C < \infty$ depending only on $\Omega$ such that for all $\psi : \cT \to \R$ satisfying
$\sum_K |K| \psi(K) = 0$, and all $\emm \in \cP(\cT)$ with $\rho(K) := \frac{\emm(K)}{|K|} \geq \delta > 0$ for all $K \in \cT$, we have
	\begin{align*}
		\sum_{K\in\cT} |K| \psi(K)^2 \leq \frac{C}{\delta} \cA_\cT(\emm,\psi) \ .
	\end{align*}
\end{proposition}

\begin{proof}
This is a straightforward modification of the proof in \cite[Lemma 3.7]{EyGaHe00}.
Define $\phi : \bOmega \to \R$ by $\phi = \sum_K \chi_K \psi(K)$ and set $ \hrho(K,L) = \theta_{KL}(\rho(K), \rho(L))$.
We need to show that
	\begin{align*}
		\frac{1}{2|\Omega|} \int_\Omega\int_\Omega |\phi(x) - \phi(y)|^2 \dd x \dd y
		  \leq \frac{C}{2 \delta} \sum_{K,L} \frac{|(K|L)|}{d_{KL}} \hrho(K,L) (\psi(K) - \psi(L))^2 \ .
	\end{align*}
For $K \sim L$ and $x, y \in \R^d$, put $\chi_{(K|L)}(x,y) = 1$ if $x,y$ belong to $\Omega$,	$(K|L)$ intersects the straight line segment connecting $x$ and $y$, and $(y-x) \cdot  (x_L - x_K) > 0$.
Otherwise, we set $\chi_{(K|L)}(x,y) = 0$.
For $K \sim L$ and $z \in \R^d$, we set
	$c_{K,L;z} := \frac{z}{|z|} \cdot \frac{x_L - x_K}{d_{KL}}$.
As $\Omega$ is convex, we obtain for a.e. $x, y \in \Omega$,
	\begin{align*}
		|\phi(x) - \phi(y)|
			 \leq \sum_{K,L} |\psi(L) - \psi(K)|  \chi_{(K|L)}(x,y)  \ .
	\end{align*}
Note that $c_{K,L;y-x} > 0$ whenever $ \chi_{(K|L)}(x,y) > 0$. Using this fact, the Cauchy--Schwarz inequality yields
	\begin{align*}
		|\phi(x) - \phi(y)|^2
		&  \leq
			\bigg(\sum_{K,L} \frac{|\psi(L) - \psi(K)|^2 }{c_{K,L;y-x}} \frac{\hrho(K,L)}{ d_{KL}}  \chi_{(K|L)}(x,y)\bigg)
		\\ & \qquad \times	\bigg(\sum_{K,L} c_{K,L;y-x} \frac{d_{KL}}{\hrho(K,L)} \chi_{(K|L)}(x,y)\bigg) \ .
	\end{align*}
For fixed $x$ and $y$, let $K_0, \ldots, K_N$ be the subsequent cells intersecting the line segment $\{(1-t) x + t y\}_{t\in [0,1]}$ as $t$ ranges from $0$ to $1$. By definition, $ \chi_{(K|L)}(x,y)$ vanishes, unless $(K|L) = (K_{i-1}|K_i)$ for some $i = 1, \ldots, N$. We thus have
	\begin{align*}
		&\sum_{K,L}c_{K,L;y-x} \frac{d_{KL}}{\hrho(K,L)}  \chi_{(K|L)}(x,y)
		   = \sum_{i=1}^N \frac{c_{K_{i-1},K_i;y-x} |x_{K_i} - x_{K_{i-1}} |}{\hrho(K_{i-1},K_i)}
		  \\& \leq \delta^{-1} \sum_{i=1}^N c_{K_{i-1},K_i;y-x} |x_{K_i} - x_{K_{i-1}} |
		   = \delta^{-1} \sum_{i=1}^N   \frac{y-x}{|y-x|} \cdot (x_{K_i} - x_{K_{i-1}})
		  \\& = \delta^{-1}  \frac{y-x}{|y-x|} \cdot (x_{K_N} - x_{K_{0}})
		   \leq \delta^{-1} R \ .
	\end{align*}
where $R =  \diam(\Omega)$. Let $B_R$ denote the ball of radius $R$ around the origin.
	Using a change of variables, we observe that
	\begin{align*}
		\int_\Omega \int_\Omega
		\frac{\chi_{(K|L)}(x,y)}{c_{K,L;y-x}}  \dd x \dd y
		& \leq \int_{B_R} \frac{1}{c_{K,L;z}} \int_\Omega\chi_{(K|L)}(x,x+z)  \dd x \dd z
		\\& \leq \int_{B_R} |(K|L)| \,|z| \dd z
		   = C   |(K|L)|
	\end{align*}
for a dimensional constant $C < \infty$.
Therefore,
	\begin{align*}
		\int_\Omega \int_\Omega & |\phi(x) - \phi(y)|^2 \dd x \dd y
	\\&  \leq  \delta^{-1} R
		  		\int_\Omega \int_\Omega  \sum_{K,L} \frac{|\psi(L) - \psi(K)|^2 }{c_{K,L;y-x}} \frac{\hrho(K,L)}{ d_{KL}}  \chi_{(K|L)}(x,y) \dd x \dd y
		 \\& \leq C \delta^{-1}   \sum_{K,L}
		 	|\psi(L) - \psi(K)|^2 \hrho(K,L) \frac{|(K|L)|}{ d_{KL}}
		 \\& = C \delta^{-1}  \cA_\cT(\emm,\psi)\ ,
	\end{align*}
for some $\Omega$-dependent constant $C < \infty$, which completes the proof.
\end{proof}

Now we are ready to prove the main result of this section.

\begin{proposition}[Comparison of the dual action functionals]\label{prop:cont-and-disc-action}
Let $\delta > 0$. For all $\mu \in \cP_\delta(\bOmega)$ and $w \in L^2(\Omega)$ with $\int_\Omega w(x) \dd x=0$ we have
	\begin{align*}
		\big|\cA^*_\cT(P_\cT\mu,P_\cT w)
		    - \bA^*(\mu,w) \big|
		       \leq C [\cT]\, \|w\|^2_{L^2(\Omega)} \ ,
	\end{align*}
where $C < \infty$ depends only on $\Omega$, $\zeta$, and $\delta$.
\end{proposition}

\begin{proof}
We use the notation from Proposition \ref{prop:error-est}.
By Remark \ref{rem:error-est} there exists a function $\phi\in H^2(\Omega)$ with $\| \phi \|_{H^2(\Omega)} \leq C \| w \|_{L^2(\Omega)}$ such that
\begin{align}
-\dive(u\nabla\phi)=w \quad \text{ and }\quad  \partial_n \phi = 0 \ .
\end{align}
Let $\emm = P_\cT\mu$ and $\sigma = P_\cT w$. As noted in Remark \ref{rem:existence} there exists $\psi : \cT \to \R$ solving
\begin{align}\label{eq:cont-eq-disc}
	- \sum_{L \in \cT} \frac{|(K|L)|}{d_{KL}} \hrho(K,L) \big(\psi(L) - \psi(K)\big)
		& = \sigma(K) \ .
	\end{align}
Recall that $\bA^*(\mu, w) = \bA(\mu, \phi)$ and $\cA_\cT^*(\emm, \sigma) = \cA_\cT(\emm, \psi)$. Using \eqref{eq:cont-eq-disc} and exploiting symmetry, we obtain
\begin{align*}
	\cA_\cT(\emm, \psi)
		  & = \frac12\sum_{K,L} \frac{|(K|L)|}{d_{KL}}
				\hrho(K,L) \big( \psi(K) - \psi(L) \big)^2
	   \\ & = \sum_K  \psi(K) \sigma(K)
       \\ & = \int_\Omega \phi w  \dd x
       		   + \bigg( \sum_{K} \bphi(K) \sigma(K)
						- \int_\Omega  \phi w \dd x \bigg)
      		   + \sum_{K} \big( \psi(K) - \bphi(K) \big) \sigma(K)
      \\ & = \bA(\mu,\phi)
      		 + \sum_{K} \int_K ( \bphi(K) - \phi) w \dd x
             - \sum_{K} e(K) \sigma(K) \ .
	\end{align*}
It remains to estimate the latter two terms.

To bound the first term, let $\dphi(K) = \dashint_K \phi \dd x$, and observe that, using  Lemma \ref{lem:trace-consequence} and the Poincar\'e inequality,
\begin{align*}
	\| \bphi(K) - \phi \|_{L^2(K)}
	& \leq	\sqrt{|K|} \, \big| \bphi(K) - \dphi(K) \big|  +
		\| \dphi(K) - \phi \|_{L^2(K)}
	\\& \leq C [\cT] \| \nabla \phi \|_{L^2(K)} \ .
\end{align*}
Therefore, the first term can be bounded by
\begin{align*}
      \bigg| \sum_{K}  \int_K  \big( \bphi(K) - \phi \big) w \dd x \bigg|
      & \leq  \sum_K \| \bphi(K) - \phi \|_{L^2(K)} \| w\|_{L^2(K)}
   \\ & \leq C [\cT] \,\|\phi\|_{H^1(\Omega)}  \|w\|_{L^2(\Omega)} \ .
	\end{align*}
To estimate the second term, we use Proposition \ref{prop:Poincare-discrete} and Proposition \ref{prop:error-est} to obtain
	\begin{align*}
	\sum_{K} e(K) \sigma(K)
	& \leq \sqrt{\sum_K \frac{\sigma(K)^2}{|K|} } \sqrt{\sum_K |K| e(K)^2}\\
    & \leq  C \|w\|_{L^2(\Omega)} \sqrt{\cA_\cT(\rho,e)}\\
    & \leq  C [\cT]\, \|w\|_{L^2(\Omega)}^2 \ .
	\end{align*}
Combining these estimates yields the result.
\end{proof}

\section{Counterexamples to Gromov--Hausdorff convergence}\label{sec:counterexamples}

In this section we show that if the asymptotic isotropy condition fails sufficiently often, then the discrete transport metric $\cW_\cT$ does \emph{not} converge to the $2$-Kantorovich metric $\bW$, in spite of the fact that the discrete heat flow converges to the continuous heat flow; see, e.g., \cite[Theorem 4.2]{EyGaHe00}. 
In fact, in the one-dimensional example below, even evolutionary $\Gamma$-convergence has been proved for the entropic gradient flow structure of the discrete heat flow with respect to the transport distance $\cW_\cT$; cf. \cite{DiLi15}.

\subsection{A one-dimensional counterexample}
\label{sec:1D-counter}

We present a one-dimensional example to illustrate the non-convergence to $\bW$ in the simplest possible setting.

We start with a well-known result on the existence of smooth $\bW$-geodesics in the one-dimensional case. For the convenience of the reader we include a direct proof. We write $\cI = [0,1]$ for brevity.

\begin{lemma}\label{lem:smooth-geodesic}
Let $\delta > 0$, and let $\mu_0, \mu_1 \in \cP_\delta(\cI)$ with densities $u_0, u_1 \in \cC^0(\cI)$ respectively.
Then there exist constants $\tilde \delta > 0$ and $C < \infty$ depending only on $\delta > 0$, such that the unique $\bW$-geodesic $(\mu_t)_{t \in [0,1]}$ connecting $\mu_0$ and $\mu_1$ satisfies $\mu_t \in \cP_{\tilde\delta}(\cI)$ and $\dd \mu_t(x) = u_t(x) \dd x$ for all $t \in [0,1]$, with $\sup_{t \in [0,1]} \| \dot u_t \|_{\cC^0(\cI)} \leq C$.
\end{lemma}

\begin{proof}
Let $F_i$ denote the distribution function of $\mu_i$ given by $F_i(x) = \int_a^x u_i(y) \dd y$, which is readily seen to be invertible. The unique optimal transport map $T$ between $\mu_0$ and $\mu_1$ is then given by $T = F_1^{-1} \circ F_0$. By the inverse function theorem, $T \in \cC^1(\cI)$ and  $T'(x) \in [M^{-1}, M]$ for all $x \in \cI$, where $M > 1$ depends on $\delta$.
The unique $\bW$-geodesic between $\mu_0$ and $\mu_1$ is given by $\mu_t = (T_t)_{\#}\mu_0$, where $T_t(x) = (1-t)x + t T(x)$, hence the density $u_t$ of $\mu_t$ satisfies
\begin{align*}
	u_t(x) = \frac{u_0(T_t^{-1}(x))}{T_t'(T_t^{-1}(x))} \ .
\end{align*}
The result follows directly from this explicit expression.
\end{proof}

For $N \in \N$ and $r \in (0,\frac12)$, we consider the $\frac{1}{N}$-periodic mesh $\cT_{r,N}$ of $\cI = [0,1]$ from Figure \ref{fig:1D}, given by
\begin{align*}
	\cT_{r,N} = \bigg\{ \bigg[ \frac{k}{N}, \frac{k+r}{N} \bigg) , \bigg[ \frac{k+r}{N}, \frac{k+1}{N} \bigg) \,:\,   0 \leq k < N - 1 \bigg \} \ .
\end{align*}
The cells in $\cT_{r,N}$ will be denoted $K_k$ for $k = 0, \ldots, 2N-1$ according to their natural ordering. To make sure that $\cT_{r,N}$ is a partition of $[0,1]$, one should add the point $1$ to the set $K_{2N-1}$, but this will be irrelevant in what follows.
Let $x_k = \frac{r+k}{2N}$ be the midpoints of $K_k$, so that $d_{k,k+1} = \frac{1}{2N}$ and $[\cT_{r,N}] = \frac{1-r}{N}$.
(For notational simplicity we write $d_{k,k+1}$ instead of $d_{K_k, K_{k+1}}$. Similarly, we write $P_{r,N}$ instead of $P_{\cT_{r,N}}$ etc.)
According to \eqref{eq:rates}, the transition rates $R_{k,k \pm 1}$ from cell $k$ to cell $k \pm 1$ are given by
\begin{align*}
	R_{k,k \pm 1}
	 = \left\{ \begin{array}{ll}
	\frac{2N^2}{1-r},
	 & \text{$k$ is odd },\\
	\frac{2N^2}{r},
	 & \text{$k$ is even } .\end{array} \right.
\end{align*}
with the understanding that $R_{0,-1} = R_{N,N+1} = 0$.

We fix an admissible mean $\theta$ (in the sense of Definition \ref{def:mean-properties}) that is assumed to be symmetric, i.e., $\theta(a,b) = \theta(b,a)$, and consider the transport metric $\cW_{r, N}$ defined by setting $\theta_{KL} = \theta$ for all $K \sim L$.
For each fixed $r \in (0,\frac12)$, the next result implies that the distances $\cW_{r,N}$ do not Gromov--Hausdorff converge to $\bW$. 
The idea of the proof is to add a suitable energy-reducing oscillation to the density of a smooth competitor; see Figure \ref{fig: 1d} below.

In Section \ref{sec:G-H} we will show that Gromov--Hausdorff convergence holds if one takes a different (non-symmetric) mean adapted to the inhomogeneity of the mesh.

\begin{proposition}\label{thm:variation-counterexample}
Fix $r\in (0,\frac12)$ and $\delta > 0$. Then there exists a constant $\eps \in (0,1)$ depending only on $r$ and $\delta$, such that for any $\mu_0, \mu_1 \in \cP_\delta(\cI)$,
    \begin{align}\label{eq:counter-1D-1}
    \limsup_{N \to \infty}
    	 \cW_{r,N}(P_{r,N} \mu_0, P_{r,N} \mu_1)
	 	 \leq (1 - \eps) \bW(\mu_0,\mu_1) \ .
    \end{align}
\end{proposition}

\begin{proof}
We divide the proof into several steps.

\smallskip
\emph{Step 1}.
Fix $r \in (0,\tfrac12)$, $\delta \in (0,1)$, and $N \geq 1$.
For $\mu \in \cP_\delta(\cI)$, set $\emm = P_{r,N} \mu$, and let $\rho$ be its density given by $\rho(K) = \frac{\emm(K)}{|K|}$.
For $\eta  \in (0, \delta)$ we define $\emm^\eta \in \cP(\cT_{r,N})$ by
\begin{align*}
	\emm^\eta(K_k) :=
	 \left\{ \begin{array}{ll}
	\emm(K_k) + \frac{r(1-r)}{N}\eta
	 & \text{$k$ is even}\ ,\\
	\emm(K_k) - \frac{r(1-r)}{N}\eta
	 & \text{$k$ is odd}\ ,\end{array} \right.
\end{align*}
so that its density is given by
\begin{align*}
	\rho^\eta(K_k) :=
	 \left\{ \begin{array}{ll}
	\rho(K_k) + (1-r) \eta
	 & \text{$k$ is even}\ ,\\
	\rho(K_k) - r \eta
	 & \text{$k$ is odd}\ .\end{array} \right.
\end{align*}
If $r$ is small, the density $\rho^\eta$ increases substantially with $\eta$ in the small (even) cells, whereas it decreases only moderately in the large (odd) cells.

We claim that, for any $\delta > 0$ and $r \in (0,\frac12)$, there exists $\eta' > 0$ and $N' < \infty$, such that for any pair of neighbouring cells $K$ and $L$, and any $\mu \in \cP_\delta(\cI)$,
\begin{align}\label{eq:scalar-inequality}
	\widehat{\rho^\eta}(K,L) \geq \widehat{\rho}(K,L) + \tfrac12 \eta (\tfrac12 - r) \ ,
\end{align}
whenever $\eta \leq \eta'$ and $N \geq N'$. Here, we write $\widehat{\rho^\eta}(K,L) = \theta(\rho^\eta(K), \rho^\eta(L))$  as usual.

\smallskip

To show this, we assume without loss of generality that $K$ is small and $L$ is large; thus $|K| = \frac{r}{N}$ and $|L| = \frac{1-r}{N}$. Define $f(\eta) := \widehat{\rho^\eta}(K,L)$.
The concavity of $\theta$ implies that
\begin{align*}
	f(\eta) \geq f(0) + \eta f'(\eta) \ ,
\end{align*}
thus it suffices to show that $f'(\eta) \geq \frac12(\frac12 - r)$.
Since $\theta$ is $1$-homogeneous, we have $\partial_1 \theta(a,a) = \frac12 = \partial_2 \theta(a,a)$ and $\partial_1 \theta(a,b) = \partial_1 \theta(a/b, 1)$ for all $a, b > 0$.
Therefore,
\begin{align*}
	f'(\eta) & = (1 - r) \partial_1 \theta(\rho^\eta(K), \rho^\eta(L))
			 -      r  \partial_2 \theta(\rho^\eta(K), \rho^\eta(L))
			 \\& = (1 - r) \partial_1 \theta\bigg(\frac{\rho^\eta(K)}{\rho^\eta(L)}, 1\bigg)
			 -      r  \partial_2 \theta\bigg(1, \frac{\rho^\eta(L)}{\rho^\eta(K)}\bigg) \ .
\end{align*}
 Set $\eps := \frac12(\frac12 - r)$ and choose $h > 0$ so small that $|\partial_1 \theta(a, 1) - \frac12| \leq \eps$ whenever $|a-1| \leq h$.
If $\eta$ and $N^{-1}$ are chosen sufficiently small (depending on $\delta$ and $r$), we obtain, since $\mu \in \cP_\delta(\cI)$,
\begin{align*}
	\bigg|\frac{\rho^\eta(K)}{\rho^\eta(L)} - 1 \bigg|
	\leq \frac{| \rho(K) - \rho(L) | + \eta}{\rho^\eta(L)}
	\leq  \frac{ N^{-1}\delta^{-1} + \eta}{\delta}
	\leq h \ ,
\end{align*}
and similarly, $\big| \frac{\rho^\eta(L)}{\rho^\eta(K)} - 1 \big| \leq h$.
Therefore,
\begin{align*}
	f'(\eta) \geq (1-r) \big(\tfrac12 - \eps\big)
					 - r \big(\tfrac12 + \eps\big)
				= \tfrac12 - r - \eps
				= \tfrac12(\tfrac12 - r) \ ,
\end{align*}
which proves the claim.

Since there is a constant $C = C(\delta) < \infty$ such that $\widehat{\rho^\eta}(K,L) \leq C$, it follows from the claim that there exists a constant $c = c(\delta) \in (0,1)$ such that
\begin{align*}
	\frac{1}{\widehat{\rho^\eta}(K,L)} \leq \frac{ 1 - c \eta(\tfrac12 - r) }{\widehat{\rho}(K,L)}
\ .
\end{align*}
Thus, for any $V : \cT_{r,N} \times \cT_{r,N} \to \R$ we have
\begin{align*}
    2N \sum_{k = 0}^{2N-1} \frac{V^2(K_k,K_{k+1})}{\hrho^\eta(K_k,K_{k+1})}
 \leq  2N \big(1 - c \eta \big(\tfrac12 - r\big)\big) \sum_{k = 0}^{2N-1} \frac{V^2(K_k,K_{k+1})}{\hrho(K_k,K_{k+1})} \ .
    \end{align*}
Using the notation from \eqref{eq:convex-action}, this means that
\begin{align}\label{eq:step-1-concl}
\cK_{r,N}(\emm^\eta,V) \leq \big(1 - c \eta \big(\tfrac12 - r\big)\big) \cK_{r,N}(\emm,V) \ .
\end{align}

\medskip
\emph{Step 2}.
Take $\mu_0, \mu_1 \in \cP_{\delta}(\cI)$ for some $\delta > 0$, and let $(\mu_t)_{t\in [0,1]}$ be the constant speed geodesic connecting $\mu_0$ and $\mu_1$.
By Lemma \ref{lem:smooth-geodesic}, there exists $\tilde \delta > 0$ such that $\mu_t \in \cP_{\tilde \delta}(\cI)$, and the density $u_t$ of $\mu_t$ satisfies $\sup_{t \in [0,1]} \|\dot u_t\|_{L^\infty(\cI)} < \infty$. It then follows from Proposition \ref{prop:cont-and-disc-action} that $\emm_t := P_{r,N} \mu_t$ satisfies
\begin{align}\label{eq:dual-action-bound}
			\cA^*_{r,N}(\emm_t, \dot\emm_t)
		    \leq \bA^*(\mu_t,\dot \mu_t) + \frac{C}{N} \ ,
\end{align}
with $C < \infty$ depending on $r$ and $\delta$.

Let $V : \cT \times \cT \times (0,1) \to \R$ be an anti-symmetric function satisfying the continuity equation \eqref{eq:cont-eq}, given by
\begin{align*}
    \dot \emm_t(K_k) + 2N \Big(V_t(K_k,K_{k+1}) - V_t(K_{k-1},K_k)\Big) = 0
\end{align*}
for all $k = 0, \ldots, 2N-1$, with $V_{0,-1} = V_{2N-1, 2N} = 0$.
Since $\dot \emm_t^\eta = \dot \emm_t$ for all $\eta$, it follows that $(\emm^\eta, V)$ solves the continuity equation as well.
Therefore, \eqref{eq:step-1-concl} yields
\begin{align*}
	\cA_{r,N}^*(\emm_t^\eta, \dot\emm_t^\eta) \leq
	\big(1 - c \eta \big(\tfrac12 - r\big)\big)
		\cA_{r,N}^*(\emm_t, \dot\emm_t)
\end{align*}
Combining this bound with \eqref{eq:dual-action-bound}, we infer that there exists a constant $c \in (0,1)$ depending only on $\tilde\delta$, such that
    \begin{align*}
    \limsup_{N \to \infty} \cW_{r,N}(\emm_0^\eta, \emm_1^\eta)
    		\leq \big(1 - c \eta\big(\tfrac12 - r\big)\big) \bW(\mu_0, \mu_1) \ ,
\end{align*}
provided $\eta$ is chosen sufficiently small depending on $\tilde\delta$ and $r$.

To finish the argument, we note that Lemma \ref{lem:distconttodisc} yields, for $i = 0,1$,
\begin{align*}
	     \cW_{r,N}(\emm_i, \emm_i^\eta)
	     & \leq C \Big( \bW( Q_{r,N} \emm_i, Q_{r,N} \emm_i^\eta )  + \frac1N \Big)
	       \leq C \Big( \frac1N\sqrt{r \eta} + \frac1N \Big)
	       \leq \frac{C}{N} \ ,
\end{align*}
where $C < \infty$ depends on $\delta$ and $r$. Consequently, by the triangle inequality,
\begin{align*}
	\limsup_{N \to \infty} \cW_{r,N}(\emm_0, \emm_1)
		\leq \big(1 - c \eta\big(\tfrac12 - r\big)\big) \bW(\mu_0, \mu_1) \ ,
\end{align*}
where $c \in (0,1)$ depends only on $\tilde\delta$. This implies \eqref{eq:counter-1D-1}.
\end{proof}

\begin{figure}[h]
    \begin{center}
        \includegraphics{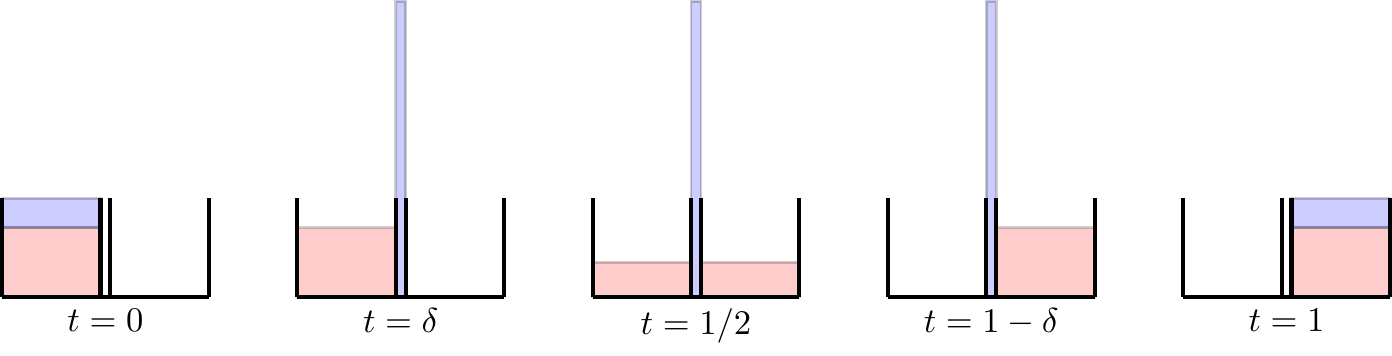}
    \end{center}
    \caption{The picture illustrates the principle behind the proof of Theorem \ref{thm:variation-counterexample}. It shows how an ``unreasonably cheap'' transport can be constructed between the measures at time $0$ and $1$. First, in the interval $[0,\delta]$, a bit of mass is moved into the short interval, so that the mean of the densities in the left and the middle cells increases. Then the bulk  of the mass is moved during the interval $[\delta,1-\delta]$, after which the short intervals are emptied in the interval $[1-\delta, 1]$. The first and the final phase are cheap because very little mass is moved, and the middle phase is cheap since the mean of the densities is kept high by the mass in the middle cell.}\label{fig: 1d}
\end{figure}

The construction in the proof of Proposition \ref{thm:variation-counterexample} breaks down if we choose mean functions $\theta_{k,k+1}$  adapted to the inhomogeneity of the grid, instead of a fixed symmetric mean $\theta$. Indeed, suppose that $\theta_{2k, 2k\pm1}$ is a smooth mean function with weight $r$ in the sense of Definition \ref{def:weight}, so that $\partial_1 \theta_{2k,2k\pm1}(1,1) = r$ and $\partial_2 \theta_{2k,2k\pm1}(1,1) = 1-r$. 
Typical examples are given in \eqref{eq:weighted-means} with $r = \lambda$.
By homogeneity of $\theta_{2k,2k\pm1}$ we have, for any $a > 0$: 
\begin{align*}
	\partial_\eta\big|_{\eta = 0}
	\theta_{2k, 2k\pm1}
		(a + (1-r) \eta, a - r \eta)
	& = (1 - r) \partial_1 \theta_{2k, 2k\pm1}
	 	(a,a)
	      - r  \partial_2 \theta_{2k, 2k\pm1}
	 	(a,a)
	= 0 \ , 
\end{align*}
hence the concave function $\eta \mapsto \theta_{2k, 2k\pm1}(a + (1-r) \eta, a - r \eta)$ attains its maximum at $\eta = 0$. 
This argument shows that one cannot increase the mean density (and thus decrease the energy) by introducing microscopic density oscillations. This is in sharp contrast to \eqref{eq:scalar-inequality}.

\FloatBarrier

\subsection{Necessity of the asymptotic isotropy condition}

Our next aim is to show that for any family of meshes $\{\cT\}$ for which the asymptotic isotropy condition fails at every scale, the distance $\cW_\cT$ is asymptotically strictly smaller than $\bW$.

We start with a lemma that guarantees the existence of certain smooth $\bW$-geodesics that transport mass in a parallel fashion.

\begin{lemma}\label{lemma:v geodesic}
Let $\Omega \subseteq \R^d$ be a bounded open set with Lipschitz boundary, let $x_0 \in \Omega$, and $v \in S^{d-1}$.
Then there exist $r > 0$, $\delta > 0$, $\kappa > 0$, and a $\bW$-geodesic $(\mu_t)_{t \in [0,1]} \subseteq \P_\delta(\bOmega)$ with the following properties: 
\begin{enumerate}[label=(\roman*)]\setlength\itemsep{0ex}
\item the continuity equation
\begin{equation}\begin{aligned}\label{eq:cont-eq-again}
\begin{cases}
    \partial_t \mu_t + \dive (\mu_t \nabla \phi_t) = 0 & \text{ in }\Omega \ , \\
    \nabla \phi_t \cdot n = 0 & \text{ on } \partial \Omega \ ,
\end{cases}
\end{aligned}\end{equation}
holds for some vector field $\phi \in \cC^1([0,1] \times \bOmega)$ satisfying $\nabla \phi_t(x) = \kappa v$ for all $t \in [0,1]$ and $x \in B(x_0, r)$;
\item $\partial_t \mu_t \in \cC^0(\bOmega)$ for all $t \in [0,1]$, and $\sup_{t \in [0,1]} \|\partial_t \mu_t\|_{\cC^0(\Omega)} < \infty$.
\end{enumerate}
\end{lemma}

\begin{proof}
Fix an open ball $B = B(x_0, r) \subseteq \Omega$ and let $\eta \in \cC_c^{\infty}(\R^d)$ be a nonnegative function, supported in the unit ball $B(0, 1)$, satisfying $\eta(x) = 1$ for $x \in B(0, \frac12)$. We define
\begin{align*}
    \phi_0(x) = v \cdot (x - x_0) \ \eta\bigg( \frac{x - x_0}{r} \bigg) \ ,
\end{align*}
so that $\phi_0 \in \cC^\infty(\bOmega)$ with support contained in $B(x_0, r)$, and $\nabla \phi_0(x) = v$ for all $x \in B(x_0, \frac{r}{2})$.

Since $\phi_0$ is smooth, there exists $T > 0$ such that the Hamilton--Jacobi equation $\partial_t \phi_t + \frac12 |\nabla \phi_t|^2 = 0$ has a unique solution in $\cC^1([0,T] \times \R^d)$ with initial value $\phi_0$. It follows from the Hopf--Lax formula $\phi_t(x) = \inf_{y} \{ \phi_0(y) + \frac{1}{2t} |x - y|^2\}$ that the following properties hold for all $t \in [0, T]$, provided $T > 0$ is sufficiently small: 
\begin{itemize}\setlength\itemsep{0ex}
\item $\supp \phi_t \subseteq B(x_0, r)$;
\item $\nabla \phi_t(x) = v$ for all $x \in B(x_0, \frac{r}{4})$.
\end{itemize}

Let $\mu_0 \in \cP(\bOmega)$ be the normalised Lebesgue measure, and set $\mu_t = (I_d + t \nabla \phi_0)_\# \mu_0$ for $t \in [0,T]$. Then $(\mu_t, \phi_t)_t$ solves the continuity equation \eqref{eq:cont-eq-again}.
Moreover, the density $\rho_t$ of $\mu_t$ solves the Monge-Amp\`ere equation
\begin{align*}
	\rho_0(x) = \rho_t(x + t \nabla \phi_0(x)) \det(I + t D^2\phi_0(x)) \ .
\end{align*}
It follows from this expression that there exist $T > 0$ and $\delta > 0$ such that $\mu_t \in \P_\delta(\Omega)$ and $\partial_t \mu_t \in \cC^0(\Omega)$ for all $t \in [0, T]$, with $\sup_{t \in [0, T]} \| \partial_t \mu_t \|_{\cC^0} < \infty$.

To obtain the result, it remains to rescale the geodesic in time. 
In doing so, we replace $\phi_t$ by $T \phi_t$, so that $\nabla \phi_t(x) = \kappa v$ for $x \in B(x_0, \frac{r}{4})$ and $t \in [0,1]$, with $\kappa = T$. 
Replacing $\frac{r}{4}$ by $r$, the result follows.
\end{proof}

The following lemma asserts that, at the macroscopic scale, the isotropy condition holds  without any assumption on the mesh. For $A \subseteq \R^d$ and $r > 0$, we let $B(A, r) := \bigcup_{x \in A} B(x, r)$ denote the $r$-neigbourhood of $A$.
 
\begin{lemma}
\label{lemma:isotropy average}
Let $\Omega \subseteq \R^d$ be a bounded convex domain. 
Let $\cT$ be an admissible mesh on $\Omega$, and let $\bflambda$ a weight function on $\cT$. 
For any open subset $U \subseteq \Omega$ and any unit vector $v \in S^{d-1}$, we have
\begin{equation}\label{eq:isotropy-average}
	\bigg| \bigg(\sum_{K, L \in \cT ; K, L \subseteq U} \lambda_{KL} 
		( v \cdot n_{KL} )^2 | (K|L)| d_{KL}\bigg) - |U| \bigg|
		\leq \big| B(\partial U, 4[\cT]) \big| \ .
\end{equation}
\end{lemma}

Note that by the symmetry of the summand, the left-hand side does not depend on the choice of the weight function $\bflambda$.

\begin{proof}
We consider the cells $C_{KL} = C_{LK} \subseteq \R^d$ defined by
\begin{align*}
	C_{KL} = \{ x \in (K|L) + \R v \subseteq \R^d
		\ : \
		x \cdot v \in \conv(x_K \cdot v, x_L \cdot v) \} \ .
\end{align*}
Observe that these sets have pairwise disjoint interiors (up to the symmetry condition $C_{KL} = C_{LK}$). 
Set $U^- = U \setminus B(\partial U, 4[\cT])$ and $U^+ =  B(U, 4[\cT])$. It then follows that  
\begin{align*}
	U^- \subseteq \bigcup_{K, L \subseteq U} C_{KL}
		\subseteq U^+ \ ,
\end{align*}
hence $|U^-| \leq \sum_{K, L \subseteq U} \lambda_{KL} |C_{KL}| \leq |U^+|$ since $\lambda_{KL} + \lambda_{LK} = 1$.
The result follows, as the area formula yields $|C_{KL}| = (v \cdot n_{KL})^2 |(K|L)| d_{KL} $.
\end{proof}
 
As the right-hand side in the previous result is small, the contribution of the term $\sum_{L \in \cT ; L \subseteq U} \lambda_{KL} ( v \cdot n_{KL} )^2 | (K|L)| d_{KL}$ is equal to $|K|$ on average, up to a microscopically small error.
However, it may happen that the isotropy condition fails at the microscopic scale, in the sense that the microscopically small error in \eqref{eq:isotropy-average} results from a cancellation of positive and negative contributions of macroscopic size. 
The following definition makes this intuition precise.

\begin{definition}[Local anisotropy]\label{def:anisotropic}
Let $\Omega \subseteq \R^d$ be a bounded convex domain, and let $U \subseteq \Omega$ be a non-empty open subset.
Let $\{\cT\}$ be a family of $\zeta$-regular meshes on $\Omega$ for some $\zeta > 0$, and for each $\cT$, let $\bflambda^\cT$ be a weight function on $\cT$. 
We say that $\{\cT\}$ is locally \emph{asymptotically $\{ \bflambda^\cT \}$-anisotropic} on $U$, if there exists a unit vector $v \in S^{d-1}$ and a constant $c > 0$, such that  
\begin{align}\label{eq:anisotropy everywhere}
\liminf_{[\cT]\to 0}\sum_{K\in \cT; K \subseteq V}
	 \bigg( \bigg( \sum_{L \in \cT; L \subseteq V} \lambda_{KL}
	 ( v \cdot n_{KL} )^2 |(K|L)| d_{KL} \bigg) - |K| \bigg)_+
	 	\geq c |V| \ ,
\end{align}
for any open cube $V \subseteq U$.
\end{definition}

\begin{example}[Anisotropy in one dimension]\label{ex:1D-anisotropy}
For $r \in (0,\frac12)$ and $N \geq 1$, consider the one-dimensional periodic mesh $\cT_{r,N}$ from Section \ref{sec:1D-counter}. We fix $s \in [0,1]$ and define a weight function $\bflambda$ on $\cT_{r,N}$ by setting $\lambda_{KL} = s$ is $K$ is small, and $\lambda_{KL} = 1 - s$ if $K$ is large. As large and small cells alternate, this indeed defines a weight function.

Fix an interval $V = (a,b)$ for some $0 < a < b < 1$, and set $v = 1$. For any $K \in \cT_{r,N}$ and $[\cT]$ sufficiently small, it follows that
\begin{align*}
 S_K := \sum_{L \in \cT_{r,N}; L \subseteq V} \lambda_{KL}
	 ( v \cdot n_{KL} )^2 |(K|L)| d_{KL} = 
	  \left\{ \begin{array}{ll}
	 \frac{s}{N}
	  & \text{if $K$ is small} \ ,\\
	 \frac{1 - s}{N}
	  & \text{if $K$ is large}\ .\end{array} \right.
\end{align*}
Note that for any neighbouring pair $K, L$, we have $S_K + S_L = \frac1N = |K| + |L|$. This means that the isotropy condition holds on average, in accordance with Proposition \ref{lemma:isotropy average}.
However, it follows that
\begin{align*}
\liminf_{N \to \infty} \sum_{K\in \cT_{r,N}; K \subseteq V}
	 \big(  S_K  - |K| \big)_+
	 	= |s-r| (b-a) \ .
\end{align*}
Therefore, the local anisotropy condition \eqref{eq:anisotropy everywhere} holds whenever $r \neq s$. If $r = s$, we have already seen in the introduction that the asymptotic isotropy condition (and in fact the centre-of-mass condition) holds.
\end{example}

\begin{example}[Anisotropy in a 2-dimensional example]\label{ex:2D-anisotropy}
Consider the crossed square grid from Figure \ref{figure:bad triangles} with $[\cT] = \frac{1}{N}$. It follows that $|(K|L)| = \frac{1}{N}$ in the coordinate directions, and $|(K|L)| = \frac{1}{N\sqrt{2}}$ in diagonal directions.
We fix $r \in (0,\frac12)$ and choose the points $x_K$ in such a way that $d_{KL} = \frac{2r}{N}$ if $n_{KL}$ points in one of the coordinate directions. If $n_{KL}$ is in  one of the diagonal directions, we then have $d_{KL} = (\frac12 - r)\frac{\sqrt{2}}{N}$.
By symmetry, it is natural to choose $\lambda_{KL} = \frac12$ for all $K \sim L$.

\newcommand{\matNS}{\left[\begin{array}{cc}0 & 0 \\0 & 1\end{array}\right]  }
\newcommand{\matSE}{\left[\begin{array}{cc}1 & 0 \\0 & 0\end{array}\right]  }

For each interior cell $K$ we compute $M_K := \sum_{L \in \cT ; L \sim K} \lambda_{KL}
 |(K|L)| d_{KL} n_{KL} \otimes n_{KL}$. Denoting the cells by $N, E, S$ and $W$,
we have
\begin{align*}
	M_N = M_S & = \frac12 \frac{2r}{N^2} \matNS 
			+ \frac12 \frac{\frac12-r}{N^2} 
				\left[\begin{array}{cc}1 & 0 \\0 & 1\end{array}\right]
	 \\ & = \frac{1}{4N^2} 
	 	\left[ \begin{array}{cc}1 - 2r & 0 \\0 & 1 + 2 r\end{array} \right] \ .
\end{align*}
An analogous computation shows that 
\begin{align*}
	M_E = M_W = \frac{1}{4N^2} \left[\begin{array}{cc}1 + 2 r & 0 \\0 & 1 - 2r\end{array}\right] \ .
\end{align*}
We thus find that
\begin{align*}
	M_N + M_E + M_S + M_W
	= \frac{1}{N^2} I 
	= \big(|N| + |E| + |S| + |W|\big) I \ ,
\end{align*}
in accordance to the fact that isotropy holds on average, for any $r \in (0,\frac12)$.

To show that the family $\{\cT_{r,N}\}_N$ is locally anisotropic, we fix $v = (1,0)$. Then:
\begin{align*}
	( v \cdot M_K v  - |K| )_+ 
	 =  \left\{ \begin{array}{ll}
	 0
	  & \text{if $K = N$ or $K = S$ }\ ,\\
	 \frac{r}{2N^2}
	  & \text{if $K = E$ or $K = W$}\ .\end{array} \right.
\end{align*}
It follows that, for any cube $V$,
\begin{align*}
\liminf_{N \to \infty} \sum_{K\in \cT_{r,N}; K \subseteq V}
	 \bigg( \bigg( \sum_{L \in \cT_{r,N}; L \subseteq V} \lambda_{KL}
	 ( v \cdot n_{KL} )^2 |(K|L)| d_{KL} \bigg) - |K| \bigg)_+
	 	= r |V| \ ,
\end{align*}
which shows that the mesh is everywhere locally anisotropic, for any $r \in (0,\frac12)$. 
\end{example}

The following proposition shows that if the mesh is locally $\{ \bflambda^\cT \}$-anisotropic, and if the mean functions $\theta^\cT$ are chosen accordingly, then the discrete transport distances are asymptotically strictly smaller than $\bW$.

\begin{theorem}[Necessity of asymptotic isotropy]
\label{thm:necessity}
Let $\Omega \subseteq \R^d$ be a bounded convex domain. 
Let $\{\cT\}$ be a family of $\zeta$-regular meshes on $\Omega$ for some $\zeta > 0$,  
and assume that $\{\cT\}$ is locally anisotropic on $U$ for some weight functions $\{ \bflambda^\cT \}$.
Let $\{\bftheta^\cT\}$ be a family of mean functions satisfying $\partial_1 \theta_{KL}^\cT(1,1) = \lambda_{KL}$ for any $\cT$ and any $K, L \in \cT$, and suppose that  the regularity condition 
\begin{equation} \label{eq:mean hessian}
\sup_{\cT} \sup_{K,L\in \cT}
	 \| D^2\theta_{KL}^\cT\|_{L^\infty(B((1,1),s))} < \infty
\end{equation}
holds for some $s > 0$.
Then there exist $\mu_0, \mu_1 \in \cP(\bOmega)$ such that
\begin{equation}
\limsup_{[\cT]\to 0} \cW_\cT(P_\cT\mu_0,P_\cT\mu_1) < \bW(\mu_0,\mu_1) \ .
\end{equation}
\end{theorem}

\begin{remark}\label{rem:exclu}
Note that all examples from Section \ref{sec:discrete-transport} satisfy \eqref{eq:mean hessian}.
However, this condition excludes certain smooth mean functions approximating $\theta(a,b) = \min(a,b)$. 
\end{remark}

\begin{proof}
We fix $x_0 \in \Omega$ and $v \in S^{d-1}$. Using Lemma \ref{lemma:v geodesic} we obtain $\delta > 0$, $r > 0$, and a geodesic $(\mu_t)_{t\in [0,1]} \subseteq \P_\delta(\bOmega)$, solving the continuity equation $\dot \mu + \dive(\mu \nabla \phi) = 0$ where the velocity vector field $\nabla \phi_t\in L^2(\Omega)$ satisfies $\nabla \phi_t(x) = \kappa v$ for some $\kappa > 0$, for all $t \in [0,1]$ and all $x$ in the ball $B = B(x_0, r)$. For brevity we write $\tilde B = B(x_0, \frac{r}{2})$.

Fix $\ell > 0$, and consider the collection of open cubes given by 
\begin{align*}
 \sQ_\ell := \{ \ell(p + (0,1)^d) \subseteq \tilde B \ : \ p \in \Z^d \} \ .
\end{align*}
For $Q \in \sQ_\ell$ we define $\cT_Q := \{ K \in \cT \ : \ K \subseteq  Q  \}$, and for $K \in \cT_Q$ we set
\begin{align*}
	S_K := \sum_{L \in \cT_Q}
		 \lambda_{KL} (v\cdot n_{KL})^2 |(K|L)| d_{KL} \ .
\end{align*}
We define the subsets $\cT_Q^+, \cT_Q^- \subseteq \cT$ by
\begin{align*}
	\cT_Q^{\pm} := \{ K \in \cT_Q \ :  (S_K - |K|)_\pm > 0 \} \ .
\end{align*}
It follows directly from \eqref{eq:anisotropy everywhere} that, for all $Q \in \sQ_\ell$, 
\begin{equation}
\label{eq:anisotropy upper}
\liminf_{[\cT]\to 0} \sum_{K \in \cT_Q^+}
		 \big( S_K -|K| \big) \geq c |Q |\ .
\end{equation}
Combining this bound with Lemma \ref{lemma:isotropy average}, we also find 
\begin{align}
\label{eq:anisotropy lower}
\liminf_{[\cT]\to 0} \sum_{K \in \cT_Q^-} 
		\big(|K| - S_K \big) \geq c |Q| \ .
\end{align}
In particular, if $[\cT]$ is sufficiently small, both $\cT_Q^+$ and $\cT_Q^-$ are non-empty.

We define a variation $\nu_\ell^\cT: \cT \to \R$ by
\begin{align*}
	\nu_\ell^\cT(K) =
		 \sum_{Q \in \sQ_\ell} 
		 \bigg( \alpha_Q^\cT \one_{\{K \in \cT_Q^+\}} - \beta_Q^\cT \one_{\{K \in \cT_Q^- \}} \bigg) \ ,
\end{align*}
where $\alpha_Q^\cT, \beta_Q^\cT \in (0,1]$ are the unique numbers such that $\sum_{K \in \cT_Q} \nu_\ell^\cT(K) |K| = 0$ and $\max \{ \alpha_Q^\cT, \beta_Q^\cT \} = 1$ for all $Q \in \sQ_\ell$.

Set $\emm_t^\cT = P_\cT \mu_t \in \cP(\cT)$, and let $\rho_t^\cT(K) = \emm_t^\cT(K)/|K|$ be its density as usual. 
We consider the perturbed measure $\emm_{\eps,t}^{\cT}$ with density $\rho_{\eps,t}^\cT$ given by 
\begin{align*}
	\rho_{\eps, t}^\cT(K) = \rho_t^\cT(K) + \eps \nu_\ell^\cT(K) \ ,
\end{align*}
where we suppress the dependence of $\rho_{\eps, t}^\cT$ on $\ell$ in the notation.
Note that $\emm_{\eps,t}^\cT$ belongs to $\cP(\cT)$ if $0 < \eps < \delta$, since $\mu_t \in \P_\delta(\bOmega)$. 
Write $\widehat{\rho_{\eps, t}^\cT}(K,L) = \theta_{KL}\left(\rho_{\eps,t}^\cT(K), \rho_{\eps,t}^\cT(L) \right)$. 
In view of the regularity assumption \eqref{eq:mean hessian} on $\theta_{KL}$ and the fact that $\mu_t \in \cP_\delta(\bOmega)$, a Taylor expansion yields
\begin{equation}\label{eq:mean Taylor expansion}
\frac{1}{\widehat{\rho_{\eps,t}^\cT}(K,L)} 
	\leq \frac{1}{\widehat{\rho_t^\cT}(K,L)} 
			- \frac{\gamma_{KL} \eps}{\widehat{\rho_t^\cT}^2(K,L)} 
			   			+ C \eps \big(\eps + [\cT]\big) 
\end{equation}
for all $\eps < \eps_0(\delta,s)$, where $\gamma_{KL} = \lambda_{KL} \nu_\ell^\cT(K) + \lambda_{LK} \nu_\ell^\cT(L)$, and $C$ depends on $\delta$ and $s$.

Let $\psi^\cT_t :\cT\to \R$ be the solution to the discrete elliptic problem \eqref{eq:elliptic-discrete}, and consider the associated momentum vector field
\begin{equation*}
V_t^\cT(K,L) = \widehat{\rho_t^\cT}(K,L) \big( \psi_t^\cT(L) - \psi_t^\cT(K) \big) \ ,
\end{equation*}
so that $\cA_\cT(\emm_t^\cT, \psi_t^\cT) = \cK_\cT(\emm_t^\cT, V_t^\cT)$.
Using \eqref{eq:mean Taylor expansion} we obtain
\begin{equation}\label{eq:three-terms}
\begin{aligned}
 \cK_\cT(\emm_{\eps,t}^\cT, V_t^\cT) 
	    \leq  \cK_\cT(\emm_t^\cT, V_t^\cT) 
       & -  \frac{\eps}{2} \sum_{K,L}  \gamma_{KL} \frac{|(K|L)|}{d_{KL}}
			\big( \psi_t^\cT(L) - \psi_t^\cT(K) \big)^2   		
	\\ & +  C  \eps \big(\eps+ [\cT]\big) 
		\sum_{K,L} \frac{|(K|L)|}{d_{KL}} V_t^\cT(K,L)^2 \  .
\end{aligned}
\end{equation}
We will estimate the three terms on the right-hand side separately.

To bound the first term, we apply Proposition \ref{prop:cont-and-disc-action} to obtain
\begin{align*}
	\cK_\cT(\emm_t^\cT, V_t^\cT) 	
	= \cA_\cT^*(\emm_t^\cT, \dot \emm_t^\cT) 	
	\leq \bA^*(\mu_t, \dot \mu_t) 
	+ C [\cT] \| \dot \mu_t\|_{L^2(\Omega)}^2 \ .
\end{align*}
Together with the uniform $L^2$-bound on $\dot \mu_t$ from Lemma \ref{lemma:v geodesic}, this implies
\begin{align}
	\label{eq:K-bound-1}
	\limsup_{[\cT]\to 0} \int_0^1 \cK_\cT(\emm_t^\cT, V_t^\cT) \dd t
	\leq \limsup_{[\cT]\to 0} \int_0^1 \bA^*(\mu_t, \dot \mu_t) \dd t
    =	 \bW(\mu_0,\mu_1)^2 \ . 
\end{align}
Moreover, as $\mu_t \in \cP_\delta(\bOmega)$, we have
\begin{align*}
	\frac12 \sum_{K,L} \frac{|(K|L)|}{d_{KL}} V_t^\cT(K,L)^2 
		\leq C \cK_\cT(\emm_t^\cT, V_t^\cT) \ ,
\end{align*}
for some $C < \infty$ depending on $\delta$. 
Using this estimate and \eqref{eq:K-bound-1} we obtain
\begin{align}
	\label{eq:K-bound-3}
	\limsup_{[\cT]\to 0} \int_0^1
		\sum_{K,L} \frac{|(K|L)|}{d_{KL}} V_t^\cT(K,L)^2  \dd t
	\leq \bW(\mu_0,\mu_1)^2 \ ,
\end{align}
which bounds the third term in \eqref{eq:three-terms}.

To treat the second term, we write $\bphi_K = \dashint_{B_K} \phi \dd x$, where $B_K = B(x_K, \zeta [\cT])$. 
Since $\sup_{K,L} \gamma_{KL} \leq 1$ and $\mu_t \in \cP_\delta(\bOmega)$, Proposition \ref{prop:error-est} yields
\begin{equation}\begin{aligned}
	\label{eq:K-bound-2a}
	& \sum_{K,L}  \gamma_{KL} \frac{|(K|L)|}{d_{KL}}
			\big( (\psi_t^\cT(L) - \psi_t^\cT(K)) - (\bphi_L  - \bphi_K) \big)^2
	\\ & \leq C \sum_{K,L}  \frac{|(K|L)|}{d_{KL}} \widehat{\rho_t^\cT}(K,L)
			\big( (\psi_t^\cT(L) - \psi_t^\cT(K)) - (\bphi_L  - \bphi_K) \big)^2
	\\ & \leq C [\cT]^2 \| \dot \mu_t \|_{L^2(\Omega)}^2 \ ,
\end{aligned}\end{equation}
where $C$ depends on $\Omega$, $\zeta$, and $\delta$.
Furthermore, for $K \in \cT_Q^{\pm}$ and $L \sim K$, we have $\nabla \phi = \kappa v$ on $K \cup L$, which implies that $\bphi_L - \bphi_K = \kappa v \cdot (x_L - x_K)$. Therefore, using the fact that $\sum_{K \in \cT} \nu_\ell^\cT (K) |K| = 0$,
\begin{align*}
	 \sum_{K,L} \gamma_{KL} \frac{|(K|L)|}{d_{KL}}
			\big( \bphi_L  - \bphi_K \big)^2   		
	  & = \frac{\kappa^2}{2} \sum_K \nu_\ell^\cT(K) \sum_L \lambda_{KL}
				 \frac{|(K|L)|}{d_{KL}}	\big( v \cdot (x_L - x_K) \big)^2
	\\& = \frac{\kappa^2}{2} \sum_K \nu_\ell^\cT(K) \big( S_K - |K| \big)
	\\& = \frac{\kappa^2}{2} \sum_{Q \in \sQ_\ell} 
			\bigg(\alpha_Q^\cT \sum_{K \in \cT_Q^+}  \big( S_K - |K| \big)
			+  \beta_Q^\cT \sum_{K \in \cT_Q^-} \big( |K| - S_K \big) \bigg) \ .
\end{align*}
Using \eqref{eq:anisotropy upper} and \eqref{eq:anisotropy lower}, this identity yields
\begin{align*}
	\liminf_{[\cT] \to 0} \sum_{K,L} \gamma_{KL} \frac{|(K|L)|}{d_{KL}}
			\big( \bphi_L  - \bphi_K \big)^2   		
	  \geq \frac{c \kappa^2}{2} \sum_{Q \in \sQ_\ell} |Q|				
	& \geq \frac{c \kappa^2}{4} |\tilde B| \ ,
\end{align*} 
provided $\ell$ is sufficiently small.
Together with \eqref{eq:K-bound-2a}, it follows that
\begin{align}\label{eq:K-bound-2}
	\liminf_{[\cT] \to 0}  \int_0^1
		\sum_{K,L}  \gamma_{KL} \frac{|(K|L)|}{d_{KL}}
			\big( \psi_t^\cT(L) - \psi_t^\cT(K) \big)^2  \dd t
			\geq \frac{c \kappa^2}{4} |\tilde B| \ .
\end{align}
Inserting the three estimates \eqref{eq:K-bound-1}, \eqref{eq:K-bound-3} and \eqref{eq:K-bound-2} into \eqref{eq:three-terms}, we obtain
\begin{align*}
	\limsup_{[\cT] \to 0} \cW(\emm_{\eps,0}^\cT, \emm_{\eps,1}^\cT)^2
	\leq 
	\limsup_{[\cT] \to 0} \int_0^1 \cK_\cT(\emm_{\eps,t}^\cT, V_t^\cT)  \dd t 
	    \leq \bW(\mu_0, \mu_1)^2
       - c  | \tilde B|  \eps +  C  \eps^2 \ ,
\end{align*}
for suitable constants $c> 0$ and $C < \infty$.

On the other hand, since $\sum_{K \in \cT_Q} \nu_\ell^\cT(K) |K| = 0$ for all $Q \in \sQ_\ell$, it follows from Lemma \ref{lem:distconttodisc} that
\begin{align*}
	\cW_\cT(\emm_0^\cT, \emm_{\eps,0}^\cT)
		& \leq C \big( \bW(\mu_0, Q_\cT \emm_{\eps,0}^\cT) + [\cT] \big)
		\\& \leq C( \ell + [\cT] ) \ ,
\end{align*}
and the same holds at $t=1$. In summary, we obtain
\begin{equation}
	\limsup_{[\cT] \to 0} \cW(\emm_0^\cT, \emm_1^\cT)^2
	    \leq \bW(\mu_0, \mu_1)^2
       - c  | \tilde B|  \eps +  C  \eps^2 + C \ell^2 \ ,
\end{equation}
for all $\eps < \eps_0$. As $\ell > 0$ is arbitrary, this yields the result.
\end{proof}

\begin{remark}\label{rem:2dcounterexample}
For the mesh in Figure \ref{figure:bad triangles}, the construction in the proof of Theorem \ref{thm:necessity} can be somewhat simplified: as discussed in Example \ref{ex:2D-anisotropy}, isotropy fails to hold in the coordinate directions $v = \pm e_1, \pm e_2$. 
Picking a $\bW$-geodesic transporting mass in direction $e_2$ in some open ball $B$, we notice that the discretisation of that geodesic transports no mass over vertical edges in $B$. The variation $\pi^\cT$ can then be set to $\pi^\cT =1$ for all $N$ and $S$ cells and $\pi^\cT = -1$ for all $E$ and $W$ cells. Along diagonal edges, the change in $\theta_{KL}$ is $o(\eps)$, whereas for horizontal edges, the change in $\theta_{N,S+e_2}$ is $-\eps + o(\eps)$. For vertical edges, the change in $\theta_{E,W+e_1}$ is $\eps + o(\eps)$, which would be costly, but here the momentum vector field $V(E,W+e_1)$ vanishes.
\end{remark}

\begin{figure}
    \begin{center}
        \includegraphics[scale = .4]{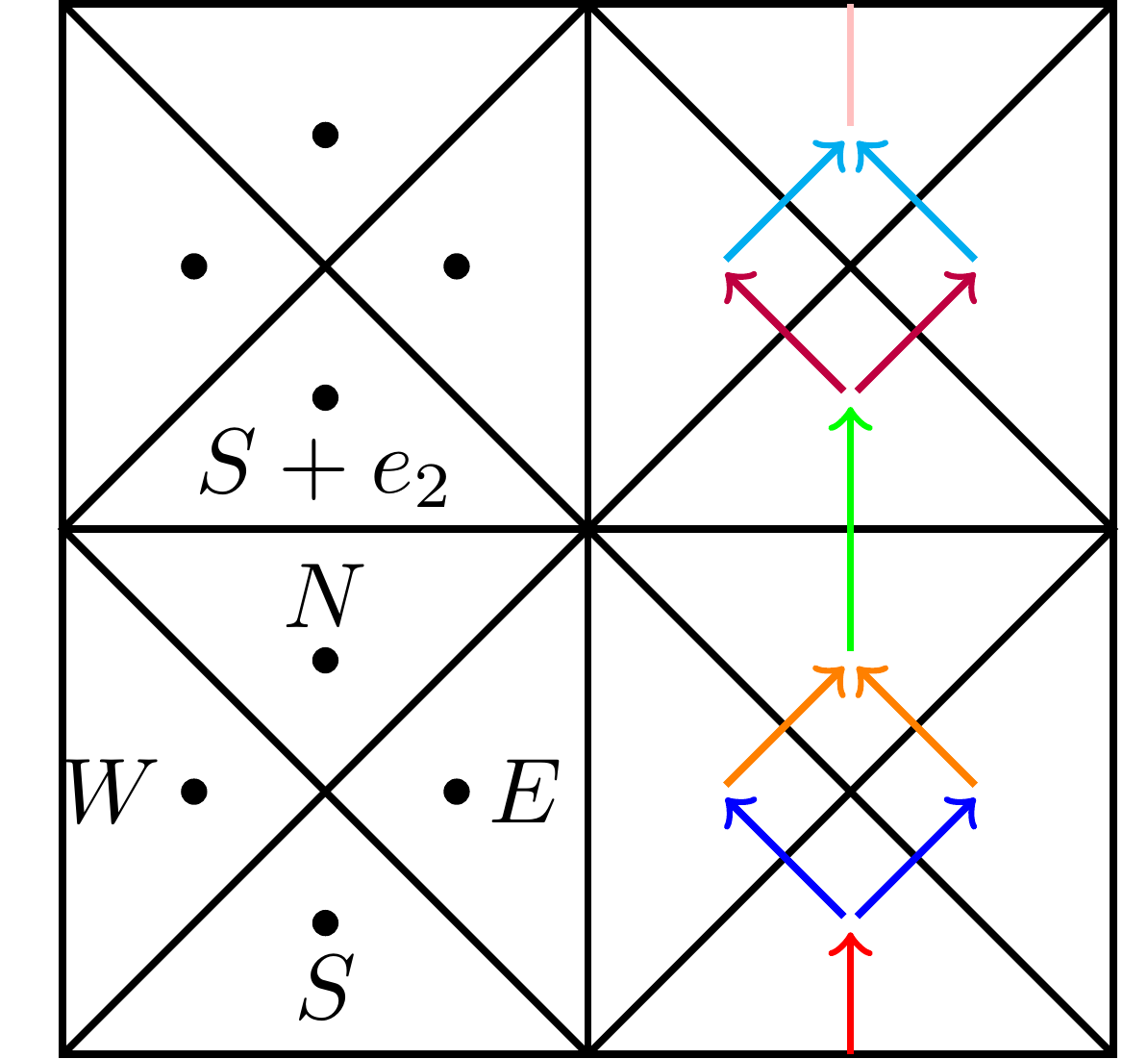}
    \end{center}
    \caption{We can decrease the action by adding $\eps$ to $\rho$ in all $N$ and $S$ triangles and subtracting $\eps$ in all $E$ and $W$ triangles. 
    If $\rho$ is the same in $E$ and $W$, we can choose a momentum vector field $V$ that transports no mass over vertical lines with $V_{SE} = V_{SW}$ and $V_{EN} = V_{WN}$.}\label{figure:bad triangles}
\end{figure}

\section{Gromov--Hausdorff convergence}\label{sec:G-H}
In this section we prove Theorems \ref{thm:upper-bound-intro} and \ref{thm:lower-bound-intro}, as well as Corollaries \ref{cor:GH-conv} and \ref{cor:min}. Let us start by stating the definition of Gromov--Hausdorff convergence.
\begin{definition}[Gromov--Hausdorff convergence]
 We say that a sequence of compact metric spaces $(\cX_n, \sd_n)_{n \geq 1}$ converges in the sense of Gromov--Hausdorff to a compact metric space $(\cX,\sd)$, if there exist  maps $f_n\colon \cX\to \cX_n$ which are
\begin{itemize}\setlength\itemsep{0ex}
\item $\eps_n$-isometric, i.e., for all $x,y\in \cX$,
\begin{equation}
 | \sd_n( f_n(x), f_n(y) ) - \sd(x,y) |\leq \eps_n
\end{equation}
and
\item $\eps_n$-surjective, i.e., for all $x_n \in \cX_n$ there exists $x \in \cX$ with
\begin{equation}
	 \sd_n(f_n(x),x_n) \leq \eps_n
\end{equation}
\end{itemize}
for some sequence $(\eps_n)_n$ with $\eps_n \to 0$ as $n \to \infty$.
\end{definition}

Our main task is to show that the mappings $\cP_\cT$ are $\eps$-isometric. We divide the argument into two parts: an upper bound for the discrete transport metric (Theorem \ref{thm:upper-bound}) will be proved in Section \ref{sec:upper}. This result is valid for any sequence of $\zeta$-regular meshes. Under strong additional symmetry assumptions, we will prove a corresponding lower bound (Theorem \ref{thm:lower-bound}) in Section \ref{sec:lower}.
The argument will be completed in Section \ref{sec:wrap-up}.

We start with a useful time-regularisation result along the lines of \cite[Lemma 2.9]{ErMa12}.

\begin{lemma}[Time-regularisation]\label{lem:smoothing}
Let $(\emm_t)_{t \in [0,1]}$ be a curve in $\cP(\cT)$ with
\begin{align*}
	\int_0^1 \cA_\cT^*(\emm_t, \dot \emm_t) \dd t < \infty \ .
\end{align*}
For $\delta \in (0, \frac12)$ we consider the ``compressed'' curve $(\emm_t^\delta)_{t \in [-\delta, 1+\delta]}$ in $\cP(\cT)$ given by
\begin{equation}
\begin{aligned}\label{compressed}
	\tilde \emm_t^\delta :=
	\begin{cases}
    \emm_0 &  \text{for } t\in[-\delta,\delta] \ ,\\
    \emm_{(t-\delta)/(1-2\delta)} &  \text{for } t\in (\delta,1-\delta) \ ,\\
    \emm_1  &  \text{for } t\in[1-\delta,1+\delta] \ .
  \end{cases}
\end{aligned}
\end{equation}
Let $\eta : \R \to \R_+$ be infinitely differentiable, symmetric, and supported in $[-1,1]$ with $\int_\R \eta(t) \dd t = 1$, and define $\emm_t^\delta := \int_{\R} \eta(s/\delta) \tilde \emm_{t-s}^\delta \dd s$. Then the following assertions hold:
\begin{enumerate}[label=(\roman*)]\setlength\itemsep{0ex}
\item The curve $(\emm_t^\delta)_{t \in [0,1]}$ is infinitely differentiable, it satisfies $\emm_0^\delta = \emm_0$ and $\emm_1^\delta = \emm_1$,  and
\begin{align*}
\int_0^1 \cA_\cT^*(\emm_t^\delta, \dot \emm_t^\delta) \dd t \leq \frac{1}{1 - 2 \delta} \int_0^1 \cA_\cT^*(\emm_t, \dot \emm_t) \dd t \ .
\end{align*}
\item Let $\{ \cT \}$ be a sequence of a meshes. For each $\cT$, let $(\emm_t^\cT)_{t\in \cT}$ be a curve in $\cP(\cT)$, and suppose that, for all $t \in [0,1]$, there exists a probability measure $\mu_t \in \cP(\bOmega)$ such that $\cQ_\cT \emm_t^\cT \weakly \mu_t$ as $[\cT] \to 0$. Then, for all $t \in [0,1]$,
\begin{align*}
	\cQ_\cT \emm_t^{\cT, \delta} \weakly \mu_t^\delta \quad \text{and} \quad
	\cQ_\cT \dot \emm_t^{\cT, \delta} \weakly \dot  \mu_t^\delta \quad \text{as } [\cT] \to 0 \ .
\end{align*}
\end{enumerate}
\end{lemma}

\begin{proof}
Using the joint convexity of the mapping $(\emm, \sigma) \mapsto \cA_\cT^*(\emm, \sigma)$ we obtain
\begin{align*}
	\int_0^1 \cA_\cT^*(\emm_t^\delta, \dot \emm_t^\delta) \dd t
		& \leq \int_0^1 \int_{-\delta}^\delta \eta(s/\delta)
			\cA_\cT^*(\tilde \emm_{t-s}^\delta,  {\dot{\tilde\emm}^\delta_{t-s}}) \dd s \dd t
		\\ & \leq  \int_0^1 \cA_\cT^*(\tilde \emm_t^\delta,  {\dot{\tilde\emm}^\delta_t}) \dd t
		= \frac{1}{1-2\delta} \int_0^1\cA_\cT^*(\emm_t, \dot \emm_t) \dd t \ .
\end{align*}
Since $\tilde \emm_t^{\cT, \delta} \weakly \tilde \mu_t^\delta$ for all $t \in [-\delta,1 + \delta]$ (where $t \mapsto \tilde \mu_t^\delta$ denotes the time-compressed version of $t \mapsto \tilde \mu_t$), the second part of the result follows using dominated convergence.
\end{proof}

Clearly, a completely analogous result holds in the continuous setting.

\subsection{Upper bound for the discrete transport metric}
\label{sec:upper}

In this subsection we prove an upper bound for $\cW_\cT$, that relies on the finite volume bounds obtained in Section \ref{sec:finite-volume}. Since these  bounds require an ellipticity condition on the densities, we use a regularisation argument involving the heat flow. Lemma \ref{lem:upper-bound} contains the desired bound for the regularised measures. The regularisation is removed in Theorem \ref{thm:upper-bound}.

We emphasise that these results hold under the mere assumption of $\zeta$-regularity, and do not require the additional symmetry assumptions that we will impose in Section \ref{sec:lower}.

\begin{lemma}\label{lem:upper-bound}
Fix $\zeta \in (0,1]$ and $a > 0$. There exists a constant $C < \infty$, depending on $\Omega$, $\zeta$, and $a$, such that for any $\zeta$-regular mesh $\cT$ of $\Omega$, and $\mu_0, \mu_1 \in \cP(\bOmega)$ the following estimate holds:
\begin{align*}
\cW_\cT(P_\cT H_a \mu_0, P_\cT H_a\mu_1)^2
\leq \bW(\mu_0, \mu_1)^2 + C [\cT] \ .
\end{align*}
\end{lemma}

\begin{proof}
Let $(\mu_t)_{t\in[0,1]}$ be a geodesic connecting $\mu_0$ and $\mu_1$.
Take $\eta : \R \to \R$ as in Lemma \ref{lem:smoothing} and define, for $\delta \in (0, \frac12)$,
\begin{align*}
	\mu_t^{a,\delta} := \int_\R \eta\Big(\frac{s}{\delta}\Big)
					H_a \tilde\mu_{t-s}^\delta \dd s \ ,
\end{align*}
where $(\tilde \mu_t^\delta)_{t \in [-\delta,1 + \delta]}$ is the compression of $(\mu_t)_{t \in [0,1]}$ as in \eqref{compressed}.
By Lemma \ref{lem:Wass-facts}, the density $u_t^{a, \delta}$ of $\mu_t^{a, \delta}$ satisfies $\Lip(u_t^{a, \delta}) \leq C$ and $u_t^{a, \delta}(x) \geq  C^{-1} > 0$ for all $x \in \bOmega$, where $C < \infty$ depends only on $a$ (and not on $t$ or $\delta$).
Proposition \ref{prop:cont-and-disc-action} yields
\begin{align}\label{eq: action estimate}
	\cA^*_\cT( P_\cT \mu_t^{a,\delta}, P_\cT \dot\mu_t^{a,\delta} )
	\leq \bA^*( \mu_t^{a,\delta}, \dot \mu_t^{a,\delta} )
		 + C [\cT] \| \dot u_t^{a, \delta} \|^2_{L^2(\Omega)} \ .
\end{align}
where $C < \infty$ depends on $\Omega$, $a$, and $\zeta$.

Denoting the density of $\tilde\mu_t^{\delta}$ by $\tilde u_t^{\delta}$, we observe that
\begin{align*}
	\dot u_t^{a, \delta} = \frac{1}{\delta} \int_\R
			\eta'\Big(\frac{s}{\delta}\Big)
					H_a \tilde u_{t-s}^\delta \dd s \ .
\end{align*}
The heat kernel upper bound \eqref{eq: heat kernel estimates} yields
\begin{align*}
	\| H_a \tilde u_t^\delta \|_{L^2(\Omega)}
			\leq C \sup_{x,y} h_a(x,y)
			\leq C (a^{-d/2} \vee 1) \ ,
\end{align*}
where $C < \infty$ depends only on $\Omega$. Consequently,
\begin{align*}
	\| \dot u_t^{a, \delta} \|_{L^2(\Omega)}
	\leq C \| \eta'\|_{L^1(\R)} (a^{-d/2} \vee 1)  \ .
\end{align*}

Integrating \eqref{eq: action estimate} over $[0,1]$ we obtain
\begin{align*}
	\cW_\cT(P_\cT H_a \mu_0, P_\cT H_a \mu_1)^2
	& \leq \int_0^1
	 \cA^*_\cT( P_\cT \mu_t^{a,\delta}, P_\cT \dot\mu_t^{a,\delta} )  \dd t
	\\& \leq \int_0^1 \bA^*( \mu_t^{a,\delta}, \dot \mu_t^{a,\delta} ) \dd t
		 + C [\cT] (a^{-d} \vee 1)
\end{align*}
where the $\eta$-dependence is absorbed in the constant $C$.
Furthermore, using the convexity of $(\mu, w) \mapsto \bA^*(\mu, w)$ as in Lemma \ref{lem:smoothing}, and the contraction bound from Lemma \ref{lem:Wass-facts}\ref{item:heat-contraction}, we obtain
\begin{align*}
	\int_0^1 \bA^*( \mu_t^{a,\delta}, \dot \mu_t^{a,\delta} ) \dd t
	 & \leq \frac{1}{1 - 2 \delta}
	 	 \int_{0}^{1} \bA^*( H_a \mu_t, H_a \dot \mu_t) \dd t
	 \\& \leq \frac{1}{1 - 2 \delta}
	 	 \int_{0}^{1} \bA^*(\mu_t, \dot \mu_t) \dd t
	   = \frac{1}{1 - 2 \delta}
	   	\bW( \mu_0, \mu_1 )^2 \ .
\end{align*}
Since the constant $C$ does not depend on $\delta$, we obtain the desired result by passing to the limit $\delta \to 0$ (and absorbing the factor $a^{-d} \vee 1$ into the constant $C$).
\end{proof}

\begin{theorem}[Upper bound for $\cW_\cT$]\label{thm:upper-bound}
Fix $\zeta \in (0,1]$. For any $\eps > 0$ there exists $h > 0$ such that for any $\zeta$-regular mesh $\cT$ with $[\cT] \leq h$, we have
\begin{align}
	\cW_\cT(P_\cT \mu_0 , P_\cT \mu_1) \leq \bW(\mu_0,\mu_1) + \eps
\end{align}
for all $\mu_0, \mu_1 \in \cP(\bOmega)$.
\end{theorem}

\begin{proof}
Using the triangle inequality we estimate
\begin{align*}
\cW_\cT(P_\cT\mu_0,P_\cT\mu_1)
	& \leq
	\cW_\cT(P_\cT \mu_0, P_\cT H_a \mu_0)
	\\& \quad	 + \cW_\cT(P_\cT H_a \mu_0,P_\cT H_a \mu_1)
		 + \cW_\cT(P_\cT H_a\mu_1,P_\cT\mu_1)
\end{align*}
for any $\mu_0, \mu_1 \in \cP(\bOmega)$ and $a>0$.
Lemma \ref{lem:upper-bound} yields
\begin{align*}
	 \cW_\cT(P_\cT H_a \mu_0, P_\cT H_a \mu_1)^2
		\leq \bW(\mu_0, \mu_1)^2
			+ C_1(a) [\cT] \ ,
\end{align*}
where $C_1(a) < \infty$ depends on $\Omega$, $\zeta$ and $a$.
Using the \emph{a priori} estimate from Lemma \ref{lem:distconttodisc} followed by Lemma \ref{lem:almost-identity} and Lemma \ref{lem:Wass-facts}, we obtain
\begin{align*}
	& \cW_\cT(P_\cT \mu_i, P_\cT H_a \mu_i)
	\\ & \leq C_2 \Big( \bW(Q_\cT P_\cT \mu_i, Q_\cT P_\cT H_a \mu_i) + [\cT] \Big)
	\\ & \leq C_2 \Big( \bW(Q_\cT P_\cT \mu_i, \mu_i)
			+ \bW(\mu_i, H_a \mu_i)
			+ \bW(H_a \mu_i, Q_\cT P_\cT H_a \mu_i) + [\cT] \Big)
	\\ & \leq C_2 \Big( [\cT] + \sqrt{a} \Big)
\end{align*}
for $i = 0, 1$, where the constant $C_2 < \infty$ depends on $\Omega$ and $\zeta$, but not on $a$.
Combining these estimates we find
\begin{align*}
\cW_\cT(P_\cT\mu_0,P_\cT\mu_1)
	\leq
	\sqrt{\bW(\mu_0,\mu_1)^2 + C_1(a) [\cT] }
			+  C_2 \Big( [\cT] + \sqrt{a} \Big)  \ .
\end{align*}

Let now $\eps > 0$, and choose $a$ sufficiently small, so that $C_2 \sqrt{a} \leq \eps / 2$. Then there exists $h > 0$ such that, whenever $[\cT] \leq h$, we have
\begin{align}\label{est1}
	\cW_\cT( P_\cT \mu_0, P_\cT \mu_1 ) \leq \bW(\mu_0, \mu_1)
			+ \eps
\end{align}
for all $\mu_0, \mu_1 \in \cP(\bOmega)$, which implies the result.
\end{proof}

\subsection{Lower bound for the discrete transport metric under isotropy conditions}
\label{sec:lower}

Since the counterexamples in Section \ref{sec:counterexamples} show that $\cW_\cT$ does not Gromov--Hausdorff converge to $\bW$ in general, we will impose an additional condition on the mesh. Let $I_d$ denote the $d \times d$ identity matrix.
\begin{definition}\label{def:asymp}
Let $\{\cT\}$ be a family of admissible meshes such that $[\cT] \to 0$.
We say that $\{\cT\}$ satisfies the \emph{asymptotic isotropy condition} with weight functions $\{\bflambda^\cT\}$ if, for all $K\in \cT$,
\begin{equation}\begin{aligned}\label{eq:asymptotic balance}
 & \sum_{L} \lambda_{KL}^\cT \frac{| (K|L) |}{d_{KL}} (x_K - x_L) \otimes (x_K - x_L) \leq |K| \big(1 + \eta_{\cT}(K) \big) I_d\ ,
\end{aligned}\end{equation}
where $\sup_{K \in \cT} |\eta_{\cT}(K)| \to 0$ as $[\cT] \to 0$.
\end{definition}

The following proposition contains the crucial bounds on the action functionals $\cA_\cT$ and their duals $\cA_\cT^*$.
To obtain the result, one needs to carefully choose the means $\theta_{KL}$ according to the geometry of the mesh.

We say that a sequence of signed measures $\{w_\cT\}_\cT \subseteq \cM_0(\bOmega)$  converges weakly to a signed measure $w$ if $\ip{\phi, w_\cT} \to \ip{\phi, w}$ for all $\phi \in \cC^0(\bOmega)$. In this case, we write $w_\cT \weakly w$.

\begin{proposition}[Action bounds]\label{prop:action-bounds}
Let $\{\cT\}$ be a family of $\zeta$-regular meshes satisfying the asymptotic \bal condition with weight functions $\{\bflambda^\cT\}$, and let $\{\bftheta^\cT\}$ be a family of   weight functions that are compatible with $\{\bflambda^\cT\}$.

Suppose that $\emm_\cT \in\cP(\cT)$ satisfies $\cQ_\cT \emm_\cT \weakly \mu$ as $[\cT] \to 0$ for some $\mu\in\mathcal\cP(\bOmega)$.
\begin{enumerate}
\item[(i)]
Let $\phi \in \cC^1(\bOmega)$ and define $\psi_\cT : \cT \to \R$ by $\psi_\cT(K) := \phi(x_K)$.
Then:
\begin{align}
\label{eq:action-limsup}
\limsup_{[\cT] \to 0}
	\cA_\cT(\emm_\cT, \psi_\cT) \leq \bA(\mu, \phi) \ .
\end{align}

\item[(ii)]
Let $\sigma_\cT \in \cM_0(\cT)$ and assume that there exists $w \in \cM_0(\bOmega)$ such that $Q_\cT \sigma_\cT \weakly w$ as $[\cT] \to 0$.
Then:
\begin{align}
\label{eq:dual-action-liminf}
	\bA^*(\mu,w) \leq \liminf_{[\cT] \to 0} \cA_\cT^*(\emm_\cT, \sigma_\cT) \ .
\end{align}
\end{enumerate}
\end{proposition}

\begin{remark}\label{rem:derivative-measure}
We emphasise that it is important to assume in \eqref{eq:dual-action-liminf} that $w$ in \eqref{eq:dual-action-liminf} is a signed measure, and not an arbitrary distribution.
\end{remark}

\begin{proof}
We will first prove \eqref{eq:action-limsup}.
For $K \in \cT$ set $v_K := \nabla \phi(x_K)$ and write $\rho_\cT(K) = {\emm_\cT(K)}/{|K|}$. Let $\omega\colon[0,\infty)\to[0,\infty)$ be the modulus of continuity of $\nabla\phi$. Then
\begin{align*}
	\big(\phi(x_K)-\phi(x_L)\big)^2
		\leq \big( v_K \cdot (x_L - x_K) \big)^2 + 4 \|\nabla\phi\|_{L^\infty} \omega(2[\cT]) d_{KL}^2
\end{align*}
whenever $L \sim K$.
By Remark \ref{rem:weighted-inequalities}, we have 
\begin{align*}
\theta^\cT_{KL}(\rho_\cT(K), \rho_\cT(L)) \leq \lambda_{KL}^\cT \rho_\cT(K) + \lambda_{LK}^\cT \rho_\cT(L) \ .
\end{align*}
Using these estimates we obtain
\begin{equation}
\begin{aligned}\label{eq: action bound}
	\cA_\cT(\emm_\cT, \psi_\cT)
		& \leq \frac12 \sum_{K,L} \frac{|(K|L)|}{d_{KL}} \Big( \lambda_{KL}^\cT \rho_\cT(K) + \lambda_{LK}^\cT \rho_\cT(L)\Big)\big( \phi(x_K) - \phi(x_L) \big)^2
		\\ & = \sum_K \rho_\cT(K) \sum_L \lambda_{KL}^\cT \frac{|(K|L)|}{d_{KL}} \big( \phi(x_K) - \phi(x_L) \big)^2
		\\ & \leq \sum_K \rho_\cT(K) \sum_L \lambda_{KL}^\cT  \frac{|(K|L)|}{d_{KL}}  \big( v_K \cdot (x_L - x_K) \big)^2
		\\& \qquad + 4\|\nabla\phi\|_{L^\infty} \omega( 2 [\cT]) \sum_K \rho_\cT(K) \sum_L d_{KL} |(K|L)| \ .
\end{aligned}
\end{equation}
In view of Lemma \ref{lem:geom-ineq}\ref{it:nbh-bound} and \eqref{eq:cell-upperbound}, we observe that
\begin{align*}
	  \sum_L d_{KL} |(K|L)|
		 \leq C  |K| \ ,
\end{align*}
where $C < \infty$ depends on $\Omega$ and $\zeta$.
In the former term we use the asymptotic isotropy condition \eqref{eq:asymptotic balance} to write
\begin{align*}
\sum_L \lambda_{KL}^\cT \frac{|(K|L)|}{d_{KL}} \big( v_K \cdot (x_L - x_K) \big)^2
    & \leq  \big(1 + \eta_\cT(K)\big)|K| \, | v_K |^2 ,
\end{align*}
where $\eta_\cT(K)$ is the error term in \eqref{eq:asymptotic balance}, which converges to $0$, uniformly in $K$, as $[\cT] \to 0$.

Summing up all contributions, we obtain
\begin{align*}
	\cA_\cT(\emm_\cT, \psi_\cT)
		\leq \sum_K \rho_\cT(K)  |K| \Big( \big(1 + \eta_\cT(K)\big) |v_K|^2
			+ 4 C \|\nabla\phi\|_{L^\infty} \omega(2[\cT])	\Big) \ .
\end{align*}
Writing $\mu_\cT := \cQ_\cT \emm_\cT$ and $\xi_\cT = \sum_{K} \chi_K \big(1 + \eta_\cT(K)\big) |v_K|^2$ we have
\begin{align*}
 &\sum_K \rho_\cT(K) |K| \big(1 + \eta_\cT(K)\big) |v_K|^2
   \\&\qquad = \ip{\xi_\cT, \mu_\cT}
 			 = \int |\nabla \phi|^2 \dd \mu
   			 + \int |\nabla \phi|^2 \dd (\mu_\cT - \mu)
			 + \int \xi_\cT - |\nabla \phi|^2 \dd \mu_\cT \ .
\end{align*}
Since $\mu_\cT$ converges weakly to $\mu$ and $\| \xi_\cT - |\nabla \phi|^2 \|_{L^\infty} \to 0$ as $[\cT] \to 0$, we obtain \eqref{eq:action-limsup}.

\medskip

Let us now prove \eqref{eq:dual-action-liminf}. Take $\phi \in \cC^1(\bOmega)$ and define $\psi_\cT : \cT \to \R$ by $\psi_\cT(K) = \phi(x_K)$.
We claim that $\ip{ \phi, w } = \lim_{[\cT] \to 0} \ip{\psi_\cT, \sigma_\cT}$. To show this, set $w_\cT := \cQ_\cT \sigma_\cT$ and $\phi_\cT := \sum_K \psi_\cT(K)\chi_K $, and  note that $\ip{\psi_\cT, \sigma_\cT} = \ip{\phi_\cT, w_\cT}$. Therefore,
\begin{align*}
	\ip{\psi_\cT, \sigma_\cT} - \ip{ \phi, w }
		= \ip{\phi_\cT - \phi, w_\cT} + \ip{ \phi, w_\cT - w } \ .
\end{align*}
Since $w_\cT \weakly w$, the Banach--Steinhaus Theorem implies that $\sup_\cT \| w_\cT \|_{\TV} < \infty$. Together with the bound $\| \phi_\cT  - \phi \|_{L^\infty} \leq C [\cT]$, this yields the claim.

Suppose first that $\bA^*(\mu, w)$ is finite.
Fix $\eps>0$ and choose $\phi \in \cC^1(\bOmega)$ such that
\begin{align*}
\frac12 \bA^*(\mu, w) \leq \ip{ \phi, w } - \frac12 \bA(\mu,\phi) + \eps \ .
\end{align*}
Using the claim and \eqref{eq:action-limsup}, it follows that
\begin{align*}
	\frac12 \bA^*(\mu, w)
	& \leq \liminf_{[\cT] \to 0}
			\Big (\ip{\psi_\cT, \sigma_\cT}
				- \frac12 \cA_\cT(\emm_\cT, \psi_\cT) \Big) +  \eps
	\\& \leq \liminf_{[\cT] \to 0}
			\frac12 \cA_\cT^*(\emm_\cT, \sigma_\cT) +  \eps \ .
\end{align*}
Since $\eps > 0$ is arbitrary, the result follows.

Suppose next that $\bA^*(\mu, w) = \infty$.
Then, for each $N > 0$, there exists a function $\phi_N \in \cC^1(\bOmega)$ such that
\begin{align*}
	\ip{\phi_N, w} - \frac12 \bA(\mu, \phi_N) \geq N \ .
\end{align*}
Define $\psi_\cT^N : \cT \to \R$ by $\psi_\cT^N(K) := \phi_N(x_K)$.
Using the claim and \eqref{eq:action-limsup} once more, we obtain
\begin{align*}
	N  & \leq \liminf_{[\cT] \to 0} \Big(\ip{\psi_\cT^N, \sigma_\cT}
				- \frac12 \cA_\cT(\emm_\cT, \psi_\cT^N) \Big)
		\\ & \leq \liminf_{[\cT] \to 0} \frac12 \cA_\cT^*(\emm_\cT, \sigma_\cT)\ ,
\end{align*}
which implies that $\liminf_{[\cT] \to 0} \cA_\cT^*(\emm_\cT, \sigma_\cT) =  \infty.$
\end{proof}

\begin{theorem}[Lower bound for $\cW_\cT$]\label{thm:lower-bound}
Fix $\zeta \in (0,1]$.
For any $\eps > 0$, there exists $h > 0$ such that the following holds: for any family of $\zeta$-regular meshes $\{\cT\}$ satisfying the asymptotic \bal condition with weight functions $\{\bflambda^\cT\}$, and any family of mean functions $\{\bftheta^\cT\}$ that are compatible with $\{\bflambda^\cT\}$, we have
\begin{align}
	\bW(\mu_0,\mu_1) \leq \cW_\cT(P_\cT \mu_0 , P_\cT \mu_1) + \eps
\end{align}
for all $\mu_0, \mu_1 \in \cP(\bOmega)$, whenever $[\cT] \leq h$.
\end{theorem}
\begin{proof}

To obtain a contradiction, we suppose that the opposite holds, i.e., there exists $\eps > 0$, a sequence of meshes $\{ \cT \}$ with $[\cT] \to 0$, and probability measures $\mu_0^\cT, \mu_1^\cT \in \cP(\cT)$, such that
\begin{align}\label{contradiction}
	\cW_\cT( P_\cT \mu_0^\cT, P_\cT \mu_1^\cT) < \bW(\mu_0^\cT, \mu_1^\cT) - \eps \ .
\end{align}
Let $(\emm_t^\cT)_{t\in [0,1]} \subseteq \cP(\cT)$ be a constant speed geodesic connecting $P_\cT \mu_0^\cT$ and $P_\cT \mu_1^\cT$, so that
\begin{align}\label{eq:geodesic-Pmu}
	\int_0^1 \cA_\cT^*(\emm_t^\cT, \dot \emm_t^\cT) \dd t
		 = \cW_\cT(P_\cT \mu_0^\cT, P_\cT \mu_1^\cT)^2 \ .
\end{align}
Set $\tilde\mu_t^\cT := H_{[\cT]} \cQ_\cT \emm_t^\cT$.
Then, for almost every $t\in[0,1]$, Lemma \ref{lem:a-priori} yields
\begin{align*}
	\bA^*(\tilde\mu_t^\cT, \dot {\tilde\mu}_t^\cT) \leq C \cA_\cT^*(\emm_t^\cT, \dot \emm_t^\cT) \ ,
\end{align*}
where the constant $C < \infty$ does not depend on $\cT$ or $t$.
Consequently, for $0 \leq t_1 \leq t_2 \leq 1$,
\begin{equation}\begin{aligned}\label{eq:equi-cont}
	\bW( \tilde\mu_{t_1}^\cT,\tilde\mu_{t_2}^\cT )
		& \leq \int_{t_1}^{t_2} \sqrt{\bA^*(\tilde\mu_t^\cT, \dot {\tilde\mu}_t^\cT)} \dd t
		\\& \leq C \int_{t_1}^{t_2} \sqrt{\cA^*_\cT(\emm_t^\cT, \dot \emm_t^\cT)} \dd t
		\leq C \sqrt{t_2-t_1}  \cW_{\cT}(P_\cT \mu_0^\cT, P_\cT \mu_1^\cT) \ ,
\end{aligned}\end{equation}
and hence by \eqref{contradiction} the family of curves $\{(\tilde\mu_{t}^\cT)_t\}_\cT$ is equicontinuous. Since $(\cP(\bOmega), \bW)$ is compact, the Arzel\`a--Ascoli Theorem yields a subsequence of meshes and a $\bW$-continuous curve of probability measures $(\lambda_t)_{t\in[0,1]}$ in $\cP(\bOmega)$ such that $\bW(\tilde\mu_t^\cT, \lambda_t) \to 0$ for all $t \in [0,1]$ as $[\cT] \to 0$.
Moreover, since
$
	\bW(\tilde\mu_t^\cT, \cQ_\cT \emm_t^\cT) \leq C \sqrt{[\cT]}
$
by Lemma \ref{lem:Wass-facts}, it follows that $\cQ_\cT \emm_t^\cT \to \lambda_t$ in $(\cP(\bOmega), \bW)$ as $[\cT] \to 0$.

For $\delta \in (0, \frac12)$, let $t \mapsto \emm_t^{\cT, \delta}$ be a time-regularised version of $t \mapsto \emm_t^\cT$, as defined in Lemma \ref{lem:smoothing}. It follows from this lemma that $\cQ_\cT \emm_t^{\cT, \delta} \weakly \lambda_t^\delta$ and $\cQ_\cT \dot \emm_t^{\cT, \delta} \weakly \dot  \lambda_t^\delta$ for all $t \in [0,1]$ as $[\cT] \to 0$. Moreover, by Lemma \ref{lem:almost-identity},
\begin{align*}
	\lambda_0^\delta = \lambda_0 = \lim_{[\cT] \to 0} \cQ_\cT \emm_0^\cT = \lim_{[\cT] \to 0} \cQ_\cT(P_\cT \mu_0^\cT) =  \lim_{[\cT] \to 0} \mu_0^\cT \ ,
\end{align*}
and similarly $\lim_{[\cT] \to 0} \mu_1^\cT  = \lambda_1^\delta$, where the convergence is with respect to $\bW$. Consequently,
\begin{align*}
	\lim_{[\cT] \to 0} \bW(\mu_0^\cT, \mu_1^\cT) = \bW(\lambda_0, \lambda_1) \ .
\end{align*}
Using Proposition \ref{prop:action-bounds} and Fatou's Lemma, Lemma \ref{lem:smoothing}, and \eqref{eq:geodesic-Pmu}, it follows that
\begin{align*}
	\bW(\lambda_0, \lambda_1)^2
		\leq  \int_0^1 \bA^*(\lambda_t^\delta, \dot \lambda_t^\delta) \dd t
		& \leq \liminf_{[\cT] \to 0}
			\int_0^1 \cA_\cT^*( \emm_t^{\cT, \delta}, \dot \emm_t^{\cT, \delta} )
					 \dd t
		\\ & \leq \frac{1}{1 - 2 \delta} \liminf_{[\cT] \to 0}  \int_0^1 \cA_\cT^*(\emm_t^\cT, \dot \emm_t^\cT) \dd t
		\\ & = \frac{1}{1 - 2 \delta} \liminf_{[\cT] \to 0} \cW_\cT(P_\cT \mu_0, P_\cT \mu_1)^2 \ .
\end{align*}
Since $\delta \in (0,\frac12)$ is arbitrary, we obtain
\begin{align*}
		\lim_{[\cT] \to 0} \bW(\mu_0^\cT, \mu_1^\cT) \leq \liminf_{[\cT] \to 0} \cW_\cT(P_\cT \mu_0, P_\cT \mu_1) \ ,
\end{align*}
which is the desired contradiction to \eqref{contradiction}.
\end{proof}

\subsection{Proof of the Gromov--Hausdorff convergence}
\label{sec:wrap-up}

It remains to prove the corollaries stated in the introduction.

\begin{proof}[Proof of Corollary \ref{cor:GH-conv}]
Fix $\eps > 0$. We will check that there exists $h > 0$ such that the map $P_{\cT}: \cP(\bOmega) \to \cP(\cT)$ is $\eps$-isometric and $\eps$-surjective whenever $[\cT] \leq h$.

The $\eps$-surjectivity holds trivially, as $\cP_\cT$ is even surjective.
To show that $P_\cT$ is $\eps$-isometric, we combine Theorem \ref{thm:upper-bound} and Theorem \ref{thm:lower-bound} to infer that there exists $h > 0$ such that
\begin{align*}
| \cW_\cT(P_\cT \mu_0, P_\cT \mu_1) - \bW(\mu_0,\mu_1)  |
	\leq \eps
\end{align*}
for all $\mu_0, \mu_1 \in \cP(\bOmega)$, whenever $[\cT] \leq h$. This yields the result.
\end{proof}

\begin{proof}[Proof of Corollary \ref{cor:min}]
For $i = 0,1$, let $\mu_i \in \cP(\bOmega)$ and $\emm_i^\cT \in \cP(\cT)$ be such that $\cQ_\cT\emm_i^\cT \weakly \mu_i$ as $[\cT] \to 0$.
Lemmas \ref{lem:distconttodisc} and \ref{lem:almost-identity} imply that, for some constant $C < \infty$ depending only on $\Omega$ and $\zeta$ (that changes from line to line),
\begin{align*}
	\cW_\cT( \emm_i^\cT, P_\cT \mu_i )
    	& \leq C \Big( \bW(\cQ_\cT \emm_i^\cT, \cQ_\cT P_\cT \mu_i) + [\cT] \Big) 
   	\\  & \leq C \Big( \bW(\cQ_\cT \emm_i^\cT, \mu_i) + [\cT] \Big) \ .
\end{align*}
As $\cQ_\cT \emm_i^\cT \weakly \mu_i$, we have $\bW(\cQ_\cT \emm_i^\cT, \mu_i) \to 0$, and therefore $\cW_\cT( \emm_i^\cT, P_\cT \mu_i ) \to 0$.
The triangle inequality then yields
\begin{align*}
	\big|\cW_\cT(\emm_0^\cT, \emm_1^\cT) 
		- \cW_\cT(P_\cT \mu_0, P_\cT \mu_1) \big|
		\leq \cW_\cT( \emm_0^\cT, P_\cT \mu_0 ) 
		+	 \cW_\cT( \emm_1^\cT, P_\cT \mu_1 ) \to 0 \ .
\end{align*}
Since $\cW_\cT(P_\cT \mu_0, P_\cT \mu_1) \to \bW(\mu_0, \mu_1)$ by Theorems \ref{thm:upper-bound-intro} and \ref{thm:lower-bound-intro}, we obtain the desired convergence $\cW_\cT(\emm_0^\cT, \emm_1^\cT) \to \bW(\mu_0, \mu_1)$. 

The final claim is now straightforward: for $0 \leq s \leq t \leq 1$ we have
\begin{align*}
	\bW(\mu_s, \mu_t) 
		= \lim_{[\cT] \to 0} \cW_\cT(\emm_s^\cT, \emm_t^\cT) 
		= (t-s) \lim_{[\cT] \to 0} \cW_\cT(\emm_0^\cT, \emm_1^\cT) 
		= (t-s) \bW(\mu_0, \mu_1) \ , 		
\end{align*}
which yields the result.
\end{proof}

\begin{remark}\label{rem:one-sided}
Clearly, it follows from the proof of Corollary \ref{cor:min} that the one-sided estimate $\limsup_{[\cT] \to 0} \cW_\cT(\emm_0^\cT, \emm_1^\cT) \leq \bW(\mu_0, \mu_1)$ holds under the conditions of Theorem \ref{thm:upper-bound-intro}, even when the asymptotic isotropy condition fails.
\end{remark}

\bibliography{literature}

\end{document}